\documentclass[11pt, a4paper]{amsart}
\usepackage{a4wide}
\usepackage{graphicx,enumitem}
\usepackage[english]{babel}
\usepackage{subfigure}
\usepackage{units}
\usepackage{diagbox}
\usepackage{color}
\usepackage{caption}
\usepackage{ulem}
\usepackage{cancel}

\usepackage{latexsym}

\usepackage[pdftex,
            pdftitle={Mean-square stability analysis of approximations of stochastic differential equations in infinite dimensions},
            pdfauthor={A. Lang, A. Petersson, and A. Thalhammer},
            bookmarksopen,
            colorlinks,
            linkcolor=black,
            urlcolor=black,
	    citecolor=black
]{hyperref}

\usepackage{dsfont}

\allowdisplaybreaks
\usepackage{mathtools}

\newcommand{\R}{\mathbb{R}}
\newcommand{\C}{\mathbb{C}}
\newcommand{\N}{\mathbb{N}}

\newcommand{\gs}{\sigma}

\newcommand{\gO}{\Omega}

\newcommand{\cA}{\mathcal{A}}

\newcommand{\cD}{\mathcal{D}} %Domains of Operators

\newcommand{\cF}{\mathcal{F}}

\newcommand{\cS}{\mathcal{S}}

\newcommand{\Ddet}{D_{\Delta t, h}^{\operatorname{det}}}
\newcommand{\Dstoch}[1]{D_{\Delta t, h}^{\operatorname{stoch},#1}}

\newcommand{\Dpart}[2]{D_{\Delta t, h}^{\operatorname{#1},#2}}
\newcommand{\kdelta}[2]{\delta_{#1, #2}}

\DeclareMathOperator{\E}{\mathbb{E}} %expectation
\DeclareMathOperator{\trace}{Tr}

\newcommand{\dd}{\mathrm{d}}

\newcommand{\inpro}[3]{ \left\langle #1 , #2 \right\rangle_{#3} }

\newcommand{\norm}[2]{\| #1 \|_{#2}}
\newcommand{\bignorm}[2]{\left\| #1 \right\|_{#2}}
\newcommand{\KL}{Karhunen--Lo\`eve } 
\newcommand{\levy}{L\'evy }

\newtheorem{lemma}{Lemma}[section]
\newtheorem{proposition}[lemma]{Proposition}

\newtheorem{theorem}[lemma]{Theorem}
\newtheorem{corollary}[lemma]{Corollary}
\theoremstyle{remark}
\newtheorem{remark}[lemma]{Remark}

\theoremstyle{definition}
\newtheorem{definition}[lemma]{Definition}
\newtheorem{assumption}[lemma]{Assumption}
\newtheorem{example}[lemma]{Example}
\begin{document}
\title[Mean-square stability analysis of approximations of SDEs in infinite dimensions]{
Mean-square stability analysis of approximations of stochastic differential equations in infinite dimensions
}

\author[A.~Lang]{Annika Lang} \address[Annika Lang]{\newline Department of Mathematical Sciences
\newline Chalmers University of Technology \& University of Gothenburg
\newline S--412 96 G\"oteborg, Sweden.} \email[]{annika.lang@chalmers.se}

\author[A.~Petersson]{Andreas Petersson} \address[Andreas Petersson]{\newline Department of Mathematical Sciences
\newline Chalmers University of Technology \& University of Gothenburg
\newline S--412 96 G\"oteborg, Sweden.} 
\email[]{andreas.petersson@chalmers.se}

\author[A.~Thalhammer]{Andreas Thalhammer} \address[Andreas Thalhammer]{\newline Doktoratskolleg Computational Mathematics \& Institute for Stochastics
\newline Johannes Kepler University Linz
\newline A--4040 Linz, Austria.} 
\email[]{andreas.thalhammer@jku.at}

\thanks{
Acknowledgement. The authors wish to express many thanks to Evelyn Buckwar and Stig Larsson for their support and for fruitful discussions and to two anonymous referees who helped to improve the results and the presentation.
The work was supported in part by the Swedish Research Council under Reg.~No.~621-2014-3995, the Knut and Alice Wallenberg foundation, and the Austrian Science Fund (FWF) under W1214-N15, project DK14. }

% \date{December 16, 2015}

\begin{abstract}
The (asymptotic) behaviour of the second moment of solutions to stochastic differential equations is treated in mean-square stability analysis. This property is discussed for approximations of infinite-dimensional stochastic differential equations and necessary and sufficient conditions ensuring mean-square stability are given. They are applied to typical discretization schemes such as combinations of spectral Galerkin, finite element, Euler--Maruyama, Milstein, Crank--Nicolson, and forward and backward Euler methods. Furthermore, results on the relation to stability properties of corresponding analytical solutions are provided. Simulations of the stochastic heat equation illustrate the theory.
\end{abstract}

\maketitle

% \todo{
% % \att{
% % }
% Idea:
% \begin{itemize}
%  \item Delete old Section~2 on S(P)DE.
%  \item Formulate Section~\ref{sec:theory} independent of SDE but as abstract numerical recursive schemes
%  \item Put in necessary SPDE theory in Section~\ref{sec:applications-galerkin}
%  \item Put (as suggested) some SPDE in the introduction as background material.
%  \item In that way, as long as the approximation scheme is linear, the theory works. This does not (in general) restrict us to linear SDE :)
% \end{itemize}
% Then it should be clear that the specific conditions on the equation are not of interest for the general mean-square stability theory that we develop.
% }
%

\section{Introduction}%\label{sec:intro}

An interesting quantity of a stochastic differential equation (SDE) or a stochastic partial differential equation (SPDE) is the qualitative behaviour of its second moment for large times. Both types of equations can be interpreted as SDEs on a (here separable) Hilbert space $(H,\inpro{\cdot}{\cdot}{H})$. More specifically, let us consider a complete filtered probability space $(\gO, \cA, (\cF_t,{t\geq 0}), P)$ satisfying the ``usual conditions'' and the model problem
\begin{equation}\label{eq:SPDE}
 \dd X(t) = (AX(t) + FX(t)) \, \dd t + G(X(t)) \, \dd L(t)
\end{equation}
with $\cF_0$-measurable, square-integrable initial condition $X(0) = X_0$.
Here, $A: \cD(A) \to H$ is the generator of a $C_0$-semigroup~$S=(S(t), t \ge 0)$ on~$H$ and $F$ is a linear and bounded operator on~$H$, i.e., $F \in L(H)$. Furthermore, $L$ denotes a $U$-valued $Q$-L\'evy process that is assumed to be a square-integrable martingale as considered in~\cite{PZ07} on the real separable Hilbert space $(U,\inpro{\cdot}{\cdot}{U})$ with covariance $Q \in L(U)$ of trace class and let $G \in L(H;L(U;H))$.

We recall from~\cite{L06} that an equilibrium (solution) of~\eqref{eq:SPDE} is the zero solution $(X_\mathrm{e}(t) = 0, t\geq 0)$.
It is called \emph{mean-square stable} if, for every $\varepsilon > 0$, there exists $\delta > 0$ such that $\E[\| X(t) \|_H^2] < \varepsilon$ for all $t\geq 0$ whenever $\E[\|X_0\|_H^2] < \delta$.
It is further \emph{asymptotically mean-square stable} if it is mean-square stable and there exists $\delta > 0$ such that $\mathbb{E}[\|X_0\|_H^2] < \delta$ implies $\lim_{t \to \infty} \E [ \| X(t) \|_H^2] = 0$.
% Furthermore, it is called \emph{asymptotically mean-square unstable} if it is not asymptotically mean-square stable.
A lot of effort has been dedicated to the asymptotic mean-square stability analysis in finite and infinite dimensions, see e.g.,~\cite{K12,M07,A74,L06}.

Since analytical solutions to SDEs are rarely available, approximations in time and possibly in space by numerical methods have to be considered. The main focus of research in recent years has been on strong and weak convergence when the discretization parameters $\Delta t$ in time and $h$ in space tend to zero. However, this property does not guarantee that the approximation shares the same (asymptotic) mean-square stability properties as the analytical solution. For finite-dimensional SDEs it is known that the specific choice of~$\Delta t$ is essential. The goal of this manuscript is to generalize the theory of asymptotic mean-square stability analysis to a Hilbert space setting. We develop a theory for approximation schemes that has apriori no relation to the original equation~\eqref{eq:SPDE} and its properties. Later on, we will discuss which conditions on~\eqref{eq:SPDE} and its approximation lead to similar behaviour.
An important application of mean-square stability is in multilevel Monte Carlo methods, where combinations of approximations on different space and time grids are computed. If the solution is mean-square unstable on any of the included levels, this is enough for the estimator to not behave as it should, see, e.g., \cite{A13}.

The mean-square stability analysis of numerical approximations of SDEs started by considering the approximations of the one-dimensional geometric Brownian motion, see e.g.,~\cite{SM96,H00a,H00}. As it has been pointed out in~\cite{BK10,BS12}, the analysis of higher-dimensional systems and their approximations is also necessary, since the asymptotic behaviour of the corresponding mean-square processes of systems with commuting and non-commuting matrices often differs.
The tools to perform mean-square stability analysis of SDE approximations presented in~\cite{BS12} could in principle be used for approximations of infinite-dimensional SDEs by a method of lines approach: After projection on an $N_h$-dimensional space the mean-square stability properties of the resulting finite-dimensional SDEs and their approximations can be determined by considering the eigenvalues of $N_h^2 \times N_h^2$-dimensional matrices. However, due to the computational complexity as $N_h \rightarrow \infty$, neither the symbolic nor the numerical computation of these eigenvalues can be done for arbitrarily large systems.
For this reason, we use an
% \annika{alternative???}
% \andreasJKU{\att{Question: extension or alternative? I would say that it is more a generalized approach compared to \cite{BS12}. The idea to consider (in finite dimensions) the Kronecker products of the involved matrices is strongly related to our tensorized linear operators (identification of tensor product via Kronecker product). So, I am not quite sure if we can call this approach an alternative...}}
approach based on tensor-product-space-valued processes and properties of tensorized linear operators.

The outline of this article is as follows:
% In Section~ref\{sec:SPDE\} we set up the framework in which mean-square stability is considered. For this, we recall basic results on linear $H$-valued SDEs driven by square-integrable, c\`adl\`ag martingales and their approximations.
Section~\ref{sec:theory} sets up a theory of mean-square stability analysis for discrete stochastic processes derived from recursions as they appear in approximations of infinite-dimensional SDEs. In the main result, necessary and sufficient conditions for asymptotic mean-square stability are shown. These results are then applied in Section~\ref{sec:applications-galerkin} to numerical approximations of~\eqref{eq:SPDE} based on spatial Galerkin discretization schemes and time discretizations with Euler--Maruyama and Milstein methods using backward/forward Euler and Crank--Nicolson as rational semigroup approximations. We conclude this work presenting simulations of stochastic heat equations with spectral Galerkin and finite element methods in Section~\ref{sec:numerics} that illustrate the theory.

\section{Asymptotic mean-square stability analysis}\label{sec:theory}

% The goal of this manuscript is to characterize mean-square stability properties of fully discrete approximations of solutions of~\eqref{eq:SPDE} such as the mild solution~\eqref{eq:SPDEmild}, which we introduce next. Therefore, 

This section is devoted to the setup of asymptotic mean-square stability for families of stochastic processes in discrete time given by recursion schemes as they typically show up in approximations of~\eqref{eq:SPDE}. We derive necessary and sufficient conditions ensuring asymptotic mean-square stability that can be checked in practice as it is shown later in Section~\ref{sec:applications-galerkin}.

Let
% \annika{$(H,\inpro{\cdot}{\cdot}{H})$ be a real separable Hilbert space and} 
$(V_h, h \in (0,1])$ be a family of finite-dimensional subspaces $V_h \subset H$ with $\operatorname{dim}(V_h) = N_h\in\mathbb{N}$ indexed by a refinement parameter~$h$.
% for the space approximation. 
With an inner product induced by $\inpro{\cdot}{\cdot}{H}$, $V_h$ becomes a Hilbert space with norm $\norm{\cdot}{H}$. For a linear operator $D: V_h \to V_h$, the operator norm $\norm{D}{L(V_h)}$ is therefore given by 
% \begin{align*}
	$\sup_{v \in V_h} \norm{D v}{H}/\norm{v}{H}$
% \end{align*}
and can be seen to coincide with $\norm{D P_h}{L(H)}$, where $P_h$ denotes the orthogonal projection onto~$V_h$.

Let us further consider the time interval $[0,\infty)$
% and let $(\gO, \cA, (\cF_t,{t\geq 0}), P)$ be a complete filtered probability space satisfying the ``usual conditions''
and for convenience equidistant time steps $t_j = j \Delta t$, $j \in \N_0$, with fixed step size $\Delta t > 0$. Hence, $t \to \infty$ is equivalent to $j\to \infty$. 
Assume that we are given a sequence of $V_h$-valued random variables $(X_h^j, j \in \N_0)$ determined by the linear recursion scheme
% For a given time point $t_j$, we denote an approximation of~$X(t_j)$ in~$V_h$ by~$X_h^j$, which is obtained by a numerical approximation scheme given by
\begin{align}\label{FullDiscretScheme}
%   \begin{split}
    X_h^{j+1} = \Ddet X_h^j + \Dstoch{j} X_h^j
%     X_h^0 &= P_h X_0.
%   \end{split}
\end{align}
 with $\cF_0$-measurable initial condition~$X_h^0 \in L^2(\gO;V_h)$, i.e., $\E [\| X_h^0 \|_{V_h}^2]< \infty$. Here $\Ddet \in L(V_h)$ and $\Dstoch{j}$ is an $L(V_h)$-valued random variable for all~$j$.
%  approximating $X_0$. 

In terms of SDE~\eqref{eq:SPDE}, one can think of $\Ddet$ as the approximation of the solution operator of the deterministic part
\begin{align*}
  \dd X(t) = (AX(t) + F X(t)) \, \dd t, \quad  t \in [t_j,t_{j+1})
\end{align*}
and $\Dstoch{j}$ approximates the stochastic part 
\begin{align*}
\dd X(t) = G(X(t)) \, \dd L(t), \quad  t \in [t_j,t_{j+1}).
\end{align*}
Although, in general, any not necessarily equidistant time discretization $(t_j, j \in \N_0)$ that satisfies $t_j \rightarrow \infty$ if $j \rightarrow \infty$ would be sufficient for the following theory, we see in the given SDE example that $\Ddet$ would be $j$-dependent in this case, which we want to omit for the sake of readability.

Inspired by properties of standard approximation schemes for~\eqref{eq:SPDE}, we put the following assumptions on the family $(\Dstoch{j}, j \in \N_0)$.

\begin{assumption}\label{ass:Dstoch}
 Let $h, \Delta t > 0$ be fixed.
 The family $(\Dstoch{j},j\in\N_0)$ is $\cF$-compatible in the sense of~\cite{BW06,L10}, i.e., $\Dstoch{j}$ is $\cF_{t_{j+1}}$-measurable and $\E[ \Dstoch{j} | \cF_{t_j} ] = 0$ for all $j \in \N_0$.
 Furthermore, for all $j \in \N_0$, let $\norm{\Dstoch{j}}{L^2(\Omega;L(V_h))} = \E[ \norm{\Dstoch{j}}{L(V_h)}^2 ]^{1/2} < \infty$ and
 \begin{equation*}
  \E\left[  \Dstoch{j} \otimes \Dstoch{j}\right| \cF_{t_j}]
    = \E\left[  \Dstoch{j} \otimes \Dstoch{j}\right],
 \end{equation*}
 where $\otimes$ denotes the tensor product.
\end{assumption}

For the recursion scheme~\eqref{FullDiscretScheme} an \emph{equilibrium} (solution) is given by the zero solution, which is defined as $X_{h,\mathrm{e}}^j = 0$ for all $j\in\N_0$. 
%Subsequently, we assume that the zero solution is the only equilibrium of~\eqref{FullDiscretScheme}. 
We define mean-square stability of the zero solution of~\eqref{FullDiscretScheme} in what follows. 
% , which can be seen as the discrete version of Definition~\ref{Def:MSStability}.
\begin{definition}%\label{Def:MSStabilityAppox}
Let $X_h= (X_h^j,j \in \N_0)$ be given by~\eqref{FullDiscretScheme} for fixed $h$ and~$\Delta t$. 
The zero solution $(X_{h,\mathrm{e}}^j = 0,j\in\mathbb{N}_0)$ of~\eqref{FullDiscretScheme} is called \emph{mean-square stable} if, for every $\varepsilon > 0$, there exists $\delta > 0$ such that
% \begin{align*}
$\E[\| X_h^j \|_H^2] < \varepsilon$  for all $j\in\N_0$
% \end{align*}
whenever $\E[\| X_h^0\|_H^2] < \delta$.

It is called \emph{asymptotically mean-square stable} if it is mean-square stable and there exists $\delta> 0$ such that $\mathbb{E}[\|X_h^0\|_H^2] < \delta$ implies
% \begin{align*}
 $\lim_{j \rightarrow \infty} \E [ \| X_h^j \|_H^2]=0$.
% \end{align*}
%
Furthermore, it is called \emph{asymptotically mean-square unstable} if it is not asymptotically mean-square stable.
\end{definition}

For convenience, the abbreviation \emph{(asymptotic) mean-square stability} for the (asymptotic) mean-square stability of the zero solution of~\eqref{FullDiscretScheme} or~\eqref{eq:SPDE} is used if it is clear from the context.

When applied to $Y_j = X^j_h$, the following lemma provides an equivalent condition for mean-square stability 
% \annika{\sout{of the zero solution of the fully discrete scheme~\eqref{FullDiscretScheme}}}
in terms of the tensor-product-space-valued process $X_h^j \otimes X_h^j \in V_h^{(2)}$. Here, for a general Hilbert space $H$, the abbreviation ${H^{(2)} = H \otimes H}$ is used and $H^{(2)}$ is defined as the completion of the algebraic tensor product with respect to the norm induced by
% \annika{\sout{ the inner product}}
\begin{align*}
	\inpro{v}{w}{H \otimes H} = \sum^N_{i = 1} \sum^M_{j = 1} \inpro{v_{1,i}}{w_{1,j}}{H} \inpro{v_{2,i}}{w_{2,j}}{H},
\end{align*}
where $v = \sum^N_{i=1} v_{1,i} \otimes v_{2,i}$ and $w = \sum^M_{j=1} w_{1,j} \otimes w_{2,j}$ are representations of elements $v$ and $w$ in the algebraic tensor product.

\begin{lemma}
\label{lem:equivalenceConvergence}
 Let $V_h$ be a finite-dimensional subspace of~$H$. Then, for any sequence $(Y_j, j \in \N_0)$ of $V_h$-valued, square-integrable random variables,
$\lim_{j \to \infty} \E [ Y_j \otimes Y_j ] = 0$
if and only if
  $\lim_{j\to \infty} \E [\| Y_j \|_{H}^2] = 0$.
\end{lemma}

\begin{proof}
% \andreas{\sout{It is clear that
%  $\lim_{j \to \infty} \E [ Y_j \otimes Y_j ] = 0$
% is equivalent to
%  $\lim_{j \to \infty} \|\E [ Y_j \otimes Y_j ]\|_{H^{(2)}} = 0$,
% since $H$ and $H^{(2)}$ are Hilbert spaces with norm induced by the inner product.} \att{Remove for the sake of brevity (if it's clear why point it out?).}}
% 
By Parseval's identity, for an orthonormal basis $(\psi_1,\dots, \psi_{N_h})$ of $V_h$, we have 
\begin{align*}
\bignorm{\E \big[ Y_j \otimes Y_j \big]}{H^{(2)}}^2  &= \sum^{N_h}_{k, \ell = 1 }  \left| \E \left[ \inpro{Y_j \otimes Y_j}{ \psi_k \otimes \psi_\ell}{H^{(2)}}\right]\right|^2 = \sum^{N_h}_{k, \ell = 1 } \left|  \E \big[ \langle Y_j, \psi_k \rangle_{H} \langle Y_j, \psi_\ell \rangle_{H}\big]  \right|^2
\end{align*}
and similarly 
\begin{align*}
\E \left[ \| Y_j \|_H^2 \right] = \sum_{k=1}^{N_h} \E \left[\inpro{Y_j}{\psi_k}{H}^2\right].
\end{align*}
Therefore, one implication is immediately obtained, while the other follows from the fact that 
\begin{equation*}
\bignorm{\E \left[ Y_j \otimes Y_j \right]}{H^{(2)}} \le \E \left[\bignorm{ Y_j \otimes Y_j }{H^{(2)}}\right] = \E \left[ \| Y_j \|_H^2 \right]. \qedhere
\end{equation*}
\end{proof}

This lemma enables us to show the following sufficient condition for asymptotic mean-square stability.

\begin{theorem}\label{thm:generalApproxMSstab}
Let $X_h = (X_h^j,j\in \N_0)$ given by~\eqref{FullDiscretScheme} satisfy Assumption~\ref{ass:Dstoch} %\todo{such that Assumption~\ref{assumption:independenceStochApprox} is satisfied}. 
and set
\begin{align*}
\cS_j = \Ddet \otimes \Ddet + \E[ \Dstoch{j} \otimes \Dstoch{j}].
\end{align*}
Then the zero solution of~\eqref{FullDiscretScheme} is asymptotically mean-square stable, if
\begin{align*}
\lim_{j\to \infty} \| \cS_j \cdots \cS_0 \|_{L(V^{(2)}_h)} = 0.
\end{align*}
\end{theorem} 

\begin{proof}
Let us first remark that $\cS_j \in L(V^{(2)}_h)$ for all $j \in \N_0$ by the properties of $\Ddet$ and $\Dstoch{j}$ and of the Hilbert tensor product.
In order to show asymptotic mean-square stability, it suffices to show $\E[X_h^j \otimes X_h^j]\to 0$ as $j \to \infty$ by Lemma~\ref{lem:equivalenceConvergence}. For this, consider
\begin{align*}
\E[X_h^{j+1} \otimes X_h^{j+1}]
&= \E\left[ (\Ddet + \Dstoch{j})X_h^j \otimes  (\Ddet + \Dstoch{j})X_h^j  \right] \\
&= \E\left[(\Ddet \otimes \Ddet) (X_h^j\otimes X_h^j)\right] + \E\left[(\Dstoch{j} \otimes \Dstoch{j})(X_h^j\otimes X_h^j)\right] \\
&\hspace{0.5cm} + \mathbb{E}\left[(\Ddet \otimes \Dstoch{j})(X_h^j \otimes X_h^j)\right] + \mathbb{E}\left[(\Dstoch{j} \otimes \Ddet)(X_h^j \otimes X_h^j)\right].
\end{align*}
% \annika{\sout{Since $\Dstoch{j}$ is independent of $\cF_{t_j}$, t}}
The mixed terms vanish by the observation that

\begin{align*}
 \mathbb{E}\left[(\Ddet \otimes \Dstoch{j})(X_h^j\otimes X_h^j)\right] 
    & = \mathbb{E}\left[(\Ddet \otimes \E[\Dstoch{j}|\cF_{t_j}])(X_h^j\otimes X_h^j)\right]
    = 0
% &= \mathbb{E}\left[\Ddet \otimes \Dstoch{j}\right]\mathbb{E}[X_h^j\otimes X_h^j] \\
% &=\left(\Ddet \otimes \mathbb{E}[\Dstoch{j}]\right)\mathbb{E}[X_h^j\otimes X_h^j] = 0,
\end{align*}
since $X_h^j$ and $\Ddet$ are $\cF_{t_j}$-measurable and $\E[\Dstoch{j}|\cF_{t_j}] = 0$ by Assumption~\ref{ass:Dstoch}.

% $\Dstoch{j}$ is $\cF$-compatible which implies that 
% \begin{align*}
% \mathbb{E}[\Dstoch{j}] = \mathbb{E}\left[\mathbb{E}[\Dstoch{j}|\cF_{t_j}]\right]=0.
% \end{align*}
Applying Assumption~\ref{ass:Dstoch} once more, we therefore conclude
% Hence, 
\begin{align*}
 \E[X_h^{j+1} \otimes X_h^{j+1}] 
    & = \E\left[\Bigl(\Ddet \otimes \Ddet + \Dstoch{j} \otimes \Dstoch{j} \Bigr) (X_h^j \otimes X_h^j)\right]\\
    & = \E\left[\Bigl(\Ddet \otimes \Ddet + \E[\Dstoch{j} \otimes \Dstoch{j}|\cF_{t_j}] \Bigr) (X_h^j \otimes X_h^j)\right]\\
    & = \Bigl(\Ddet \otimes \Ddet + \E[\Dstoch{j} \otimes \Dstoch{j}] \Bigr)\E\left[X_h^j \otimes X_h^j)\right].
\end{align*}
and obtain
% By once again noting that $\Dstoch{j}$ is independent of~$\cF_{t_j}$ and that the linear operator $\Ddet$ is deterministic, we get
\begin{align*}
\E[X_h^{j+1} \otimes X_h^{j+1}] 
  = \cS_j \E[X_h^j \otimes X_h^j] 
  = (\cS_j \cdots \cS_0) \E[X_h^0 \otimes X_h^0].
\end{align*}
% \annika{The assumption} 
Since $\lim_{j\to \infty }\| \cS_j \cdots \cS_0\|_{L(V_h^{(2)})}=0$, mean-square stability is shown with the computation
% implies \annika{the existence of} a constant~$C$ such that 
% \begin{align*}
% $\sup_{j\in\mathbb{N}} \| \cS_j \cdots \cS_0\|_{L(V_h^{(2)})}\leq C$,
% \end{align*}
% \annika{which yields mean-square stability since}
\begin{align*}
\mathbb{E}[\| X_h^{j+1} \|_H^2]^2 
  &= \Bigl(\sum_{k=1}^{N_h} \E [ \langle X_h^{j+1},\psi_k \rangle^2_H] \Bigr)^2 
  \leq  N_h \sum_{k=1}^{N_h} \E [ \langle X_h^{j+1},\psi_k \rangle^2_H]^2 \\
  &\leq N_h  \big\| \mathbb{E}[X_h^{j+1} \otimes X_h^{j+1} ] \big\|_{H^{(2)}}^2 
  \leq N_h \| \cS_j \cdots \cS_0 \|_{L(V_h^{(2)})}^2 \mathbb{E}[\|X_h^0 \|_H^2]^2. %\\
%   &\leq C^2 N_h \,\mathbb{E}[\|X_h^0 \|_H^2]^2.
\end{align*}
% This implies mean-square stability.
% 
For asymptotic mean-square stability, note that for any $\cF_0$-measurable initial value $X_h^0\in L^2(\Omega;V_h)$ it holds that $\lim_{j\to \infty} \E[X_h^j \otimes X_h^j] = 0$ if and only if
\begin{align*}
\lim_{j\to \infty} \| (\cS_j \cdots \cS_0) \E[X_h^0 \otimes X_h^0] \|_{V_h^{(2)}}
  =0,
\end{align*}
for which a sufficient condition is given by
% \begin{align*}
$\lim_{j\to \infty} \| \cS_j \cdots \cS_0 \|_{L(V^{(2)}_h)} = 0$.
% \end{align*}
Therefore, the proof is finished.
\end{proof}

In many examples the operators $(\Dstoch{j},j\in\N_0)$ have a constant covariance, i.e., they satisfy for all $j \in \N_0$
\begin{align}\label{eq:stationary_covariance}
\E\left[  \Dstoch{j} \otimes \Dstoch{j}\right] = \E\left[  \Dstoch0 \otimes \Dstoch0  \right].
\end{align}
Often as in the following example they are even independent and identically distributed, which implies~\eqref{eq:stationary_covariance}.

\begin{example}
Consider the one-dimensional geometric Brownian motion driven by an adapt\-ed, real-valued Brownian motion $(\beta(t),t\geq 0)$ 
\begin{align*}
\dd X(t) = \lambda X(t) \dd t + \sigma X(t) \dd \beta(t), \qquad t\geq 0,
\end{align*}
with initial condition $X(0) = x_0 \in \R$ and $\lambda,\sigma\in\R$.
The solution can be approximated by the explicit Euler--Maruyama scheme
\begin{align*}%\tag{EM}
X_{j+1} &= X_j + \lambda \Delta t X_j + \sigma \Delta \beta^j X_j,
\end{align*}
for $j \in \N_0$, where $\Delta \beta^j = \beta(t_{j+1}) - \beta(t_j)$, or by the Milstein scheme
\begin{align*}%\tag{Mil}
X_{j+1} = X_j + \lambda \Delta t X_j + \sigma \Delta \beta^j X_j + 2^{-1} \sigma^2 \left((\Delta \beta^j)^2 - \Delta t\right) X_j.
\end{align*}
Then the deterministic operators in~\eqref{FullDiscretScheme}
 \begin{align*}
  D_{\Delta t, \text{EM}}^{\operatorname{det}} 
    = D_{\Delta t, \text{Mil}}^{\operatorname{det}} 
    = 1 + \lambda \Delta t
 \end{align*}
are equal for both schemes, and the corresponding approximations of the stochastic integrals are given by
\begin{align*}
D_{\Delta t, \text{EM}}^{\operatorname{stoch},j} = \sigma \Delta \beta^j, \qquad
D_{\Delta t, \text{Mil}}^{\operatorname{stoch},j} = \sigma \Delta \beta^j + 2^{-1} \sigma^2 \left((\Delta \beta^j)^2 - \Delta t\right)
\end{align*}
for $j \in \N_0$.
Both families of stochastic approximation operators satisfy Assumption~\ref{ass:Dstoch}
% \sout{are $\cF$-compatible and consist of identically distributed linear operators satisfying our assumption that $\Dstoch{j}$ is independent of\annika{\sout{ the filtration}}~\andreas{$\cF_{t_j}$} since it is admissible} 
and
\begin{align*}
 \E\left[D_{\Delta t, \text{EM}}^{\operatorname{stoch},j} \otimes D_{\Delta t, \text{EM}}^{\operatorname{stoch},j}\right]
  = \gs^2 \Delta t,
  \qquad
  \E\left[D_{\Delta t, \text{Mil}}^{\operatorname{stoch},j} \otimes D_{\Delta t, \text{Mil}}^{\operatorname{stoch},j}\right]
  = \gs^2 \Delta t\left(1 + 2^{-1} \gs^2 \Delta t\right)
\end{align*}
do not depend on~$j$.
We observe that the equidistant time step~$\Delta t$ is essential here.
\end{example}

Having this example in mind, we are able to give a necessary and sufficient condition for asymptotic mean-square stability when assuming~\eqref{eq:stationary_covariance} and therefore to specify Theorem~\ref{thm:generalApproxMSstab}. The condition relies on the spectrum of a single linear operator $\cS \in L(V^{(2)}_h)$.

\begin{corollary}\label{cor:generalApproxMSstabiid} 
Let $X_h = (X_h^j,j\in \N_0)$ given by~\eqref{FullDiscretScheme} satisfy Assumption~\ref{ass:Dstoch} and %\todo{Assumption~\ref{assumption:independenceStochApprox} is satisfied and} 
% $(\Dstoch{j},j\in\N_0)$ 
\eqref{eq:stationary_covariance}.
Then the zero solution of~\eqref{FullDiscretScheme} is asymptotically mean-square stable if and only if
\begin{align*}
\cS = \Ddet \otimes \Ddet + \E[ \Dstoch{0} \otimes \Dstoch{0}] \in L(V^{(2)}_h)
\end{align*}
satisfies $\rho( \cS) = \max_{i=1,\dots,N_h^2} |\lambda_i| < 1$, where $\lambda_1,\dots,\,\lambda_{N_h^2}$ are the eigenvalues of~$\cS$.
% $\rho(\cS)$ is the spectral radius of $\cS$ defined as $\rho(\cS) = \max_{i=1,\dots,N_h^2} |\lambda_i|$ for $\lambda_1,\dots,\,\lambda_{N_h^2}$ being the eigenvalues of the operator~$\cS$. 

Furthermore, it
% for the zero solution of the discrete approximation scheme~\eqref{FullDiscretScheme} being asymptotically mean-square stable
is asymptotically mean-square stable if $\norm{\cS}{L(V_h^{(2)})} < 1$.
\end{corollary}

\begin{proof}
% First note that due to the identical distribution of the operators $(\Dstoch{j},j\in\N_0)$, it is clear that
% \begin{align*}
% \E\left[  \Dstoch{j} \otimes \Dstoch{j}\right] = \E\left[  \Dstoch0 \otimes \Dstoch0  \right].
% \end{align*}
Setting $\cS_j = \cS$ for all $j \in \N_0$ in Theorem~\ref{thm:generalApproxMSstab}, we obtain by the same arguments
\begin{align*}
\E[X_h^{j+1} \otimes X_h^{j+1}] = \left(\cS_j  \cdots  \cS_0\right)\E[X_h^0 \otimes X_h^0] = \cS^{j+1} \E[X_h^0 \otimes X_h^0].
\end{align*} 
% by the same arguments as in Theorem~\ref{thm:generalApproxMSstab}.
As a consequence, $\lim_{j\to \infty} \E[X_h^j \otimes X_h^j] = 0$ if and only if 
$\lim_{j\to \infty} \cS^{j} = 0$
which is equivalent to $\rho(\cS) < 1$ by the same arguments as, e.g., in~\cite{B64,K12,BS12}. This completes the proof of the first statement.
Since $\rho(\cS) \leq \| \cS \|_{L(V_h^{(2)})}$, a sufficient condition for asymptotic mean-square stability is given by $\norm{\cS}{L(V_h^{(2)})} < 1$.
\end{proof}

In the framework of SDE approximations, note that this corollary is an SPDE version formulated with operators of
% extends \att{\andreas{There was the question of whether this qualified as an extension or not.}} 
the results for finite-dimensional linear systems in~\cite{BS12}.
% to the case of SPDE approximations on an operator-valued level of consideration. 
There, the proposed method relies on a matrix eigenvalue problem. For SPDE approximations, this approach is not suitable, since the dimension of the considered eigenvalue problem increases heavily with space refinement. More precisely, for $h>0$, the spectral radius of an $(N_h^2 \times N_h^2)$-matrix has to be computed. To overcome this problem, we perform, in what follows, mean-square stability analysis of SPDE approximations based on operators as introduced above.

\section{Application to Galerkin methods}
\label{sec:applications-galerkin}

We continue by applying the previous results to the analysis of some classical numerical approximations of~\eqref{eq:SPDE}
which admits by results in~\cite[Chapter 9]{PZ07} an up to modification unique mild c\`adl\`ag solution and is for $t \ge 0$ given by
\begin{equation}\label{eq:SPDEmild}
X(t) = S(t) X_0 + \int^t_0 S(t-s)F(X(s)) \, \dd s+ \int^t_0 S(t-s)G(X(s)) \, \dd L(s).
\end{equation}
We assume further that the operator $-A: \cD(-A) \subset H \to H$ of~\eqref{eq:SPDE} is densely defined, self-adjoint, and positive definite with compact inverse.
% The considered methods are based on a spatial discretization using Galerkin methods combined with different one-step time integration schemes such as one-step Euler--Maruyama or Milstein methods. In order to use the analytic framework of \cite{K14}, we assume the operator $-A: \cD(-A) \subset H \to H$ of~\eqref{eq:SPDE} to be densely defined, self-adjoint, and positive definite with compact inverse.
This implies that $-A$ has a non-decreasing sequence of positive eigenvalues ${(\lambda_i, i \in \N)}$ for an orthonormal basis of eigenfunctions $(e_i, i \in \N)$ in $H$ and fractional powers of $-A$ are provided by
\begin{align*}
(-A)^{r/2} e_i = \lambda_{i}^{r/2} e_i
\end{align*}
for all $i \in \N$ and $r > 0$. For each $r>0$, $\dot{H}^r = \cD((-A)^{r/2})$ with inner product $\inpro{\cdot}{\cdot}{r}=\inpro{(-A)^{r/2}\cdot}{(-A)^{r/2}\cdot}{H}$ defines a separable Hilbert space (see, e.g.,~\cite[Appendix B]{K14}).
% Let us further from here on assume that $G \in L(H;L(U;H))$, \andreas{which, if $Q$ is trace-class, is stronger than the assumption $G \in L(H;L_{\text{HS}} (Q^{1/2}(U);H))$ since $\|G\|_{L(H;L_{\text{HS}} (Q^{1/2}(U);H))} \le \sqrt{\trace(Q)} \| G \|_{L(H;L(U;H))} $.}

% For the spatial approximation of the solution of~\eqref{eq:SPDE} we employ a Galerkin method. For this, 
Let the sequence $(V_h, h \in (0,1])$ of finite-dimensional subspaces fulfil $V_h\subset\dot{H}^1\subset H$ 
%We denote by $P_h$ the orthogonal projection onto~$V_h$.
and define the discrete operator $-A_h : V_h \to V_h$ by
% is defined on each $v_h \in V_h$ by letting $-A_h v_h$ be the unique element of $V_h$ such that
\begin{align*}
 \inpro{-A_h v_h}{w_h}{H} = \inpro{v_h}{w_h}{1} = \inpro{(-A)^{1/2} v_h}{(-A)^{1/2} w_h}{H}
\end{align*}
for all $v_h, w_h \in V_h$. This definition implies that $-A_h$ is self-adjoint and positive definite on $V_h$ and therefore has a sequence of orthonormal eigenfunctions $(e_{h,i}, i = 1,\ldots,N_h)$ and positive non-decreasing eigenvalues $(\lambda_{h,i}, i = 1,\ldots,N_h)$ (see e.g.,~\cite[Chapter 3]{K14}). 
By using basic properties of the Rayleigh quotient, we bound the smallest eigenvalue $\lambda_{h,1}$ of $-A_h$ from below by the smallest eigenvalue $\lambda_1$ of $-A$ through
\begin{align}
\label{eq:eigenvalue_inequality}
\lambda_{h,1} = \min_{v_h \in V_h  \setminus \{ 0 \}} \frac{\inpro{v_h}{v_h}{1}}{\norm{v_h}{H}^2} \ge \min_{v \in H \setminus \{ 0 \}} \frac{\inpro{v}{v}{1}}{\norm{v}{H}^2} = \lambda_1,
\end{align}
since $V_h \subset H$, cf.~\cite{B10}. This estimate turns out to be useful when comparing asymptotic mean-square stability of~\eqref{eq:SPDE} and its approximation later in this section.

Let the covariance of the \levy process~$L$ be self-adjoint, positive semidefinite, and of trace class. Then results in~\cite[Chapter~4]{PZ07} imply the existence of an orthonormal basis $(f_i, i \in \N)$ of~$U$ and a non-increasing sequence of non-negative real numbers $(\mu_i, i \in \N)$ such that for all $i \in \N$, $Q f_i = \mu_i f_i$ with $\trace(Q) = \sum_{i=1}^\infty \mu_i < \infty$ and $L$ admits a \KL expansion
\begin{equation}
\label{eq:karhunenloeveexpansion}
L(t) = \sum^\infty_{i=1} \sqrt{\mu_i} L_i(t) f_i,
\end{equation}
where $(L_i, i \in \N)$ is a family of real-valued, square-integrable, uncorrelated \levy processes satisfying $\E[(L_i(t))^2]=t$ for all $t \ge 0$.
Note that due to the martingale property of~$L$, the real-valued \levy processes satisfy $\E[L_i(t)] = 0$ for all $t \ge 0$ and $i \in \N$. This implies, together with the stationarity of the \levy increments $\Delta L^j_i = L_i(t_{j+1})-L_i(t_j)$, that for all $i \in \N$ and $j \in \N_0$,
\begin{align*}%\label{Eq:Mean0Levy}
\E[\Delta L_i^j] = \E[\Delta L_1^0] = \E[L_1(\Delta t)] = 0.
\end{align*}
Since the series representation of $L$ can be infinite, an approximation of~$L$ might be required to implement a fully discrete approximation scheme, which is typically done by truncation of the \KL expansion, i.e., for $\kappa \in \mathbb{N}$, set 
% \begin{align*}
$L^\kappa(t) = \sum_{i=1}^\kappa \sqrt{\mu_i} L_i(t) f_i$.
% \end{align*}
Note that the choice of $\kappa$ is essential and should be coupled with the overall convergence of the numerical scheme as is discussed in, e.g.,~\cite{BL12,BL12c,LPS14}. Within this work, we consider numerical methods based on the original \KL expansion~\eqref{eq:karhunenloeveexpansion} of~$L$. However, this does not restrict the applicability of the results since $L^\kappa$ fits in the framework by setting $\mu_i = 0$ for all $i > \kappa$.

As standard example in this context we consider the stochastic heat equation which is used for simulations in Section~\ref{sec:numerics}. 
\begin{example}[Stochastic heat equation] \label{Ex:stochastic_heat_equation}
% 	\label{Ex:stochastic_heat_equation}
Let the separable Hilbert space $H=L^2([0,1])$ be the space of square-integrable functions on $[0,1]$. On this space we consider the operator $A=\nu \Delta$, where $\nu > 0$ and  $\Delta$ denotes the Laplace operator with homogeneous zero Dirichlet boundary conditions which is the generator of a $C_0$-semigroup, cf.\ \cite[Example 2.21]{K14}. 
% Furthermore, let the square-integrable martingale $M=L$ be a $U$-valued \levy process. 
The equation 
\begin{align*}
 \dd X(t) = \nu \Delta X(t) \, \dd t + G(X(t)) \, \dd L(t)
\end{align*}
is referred to as the (homogeneous) stochastic heat equation.
% \end{example}

% \begin{example} \label{Ex:stochastic_heat_equation_G}
% Consider the setting of Example~\ref{Ex:stochastic_heat_equation}.
It is known (see, e.g.,~\cite[Chapter 6]{K14}) that the eigenvalues and eigenfunctions of the operator $-A$ are given by
\begin{align*}
\lambda_i = \nu i^2\pi^2, \quad e_i(y) = \sqrt{2} \sin (i \pi y), \qquad i \in \N, y\in[0,1].
\end{align*}
We first assume, for simplicity, that $U=H=L^2([0,1])$ and that the operator $Q$ diagonalizes with respect to the eigenbasis of $-A$, i.e., $f_i = e_i$ for all $i \in \N$. For this choice, we consider the operator $G=G_1$ that gives rise to a \emph{geometric Brownian motion in infinite dimensions}, cf.\ \cite[Section 6.4]{K14}. It is for all $u,v \in H$ defined by the equation
\begin{align*}%\label{eq:infiniteGBMDiffusion}
G_1(v)u = \sum_{i=1}^\infty \langle v,e_i \rangle_H \langle u,e_i \rangle_H e_i .
\end{align*}
As a second example, we let $U=\dot{H}^1$ with the same diagonalization assumption as before, i.e., $f_i = \lambda^{1/2}_i e_i$ for all $i \in \N$. Here, we let the operator $G=G_2$ be a \emph{Nemytskii operator} which is defined pointwise for $x \in [0,1]$, $u \in \dot{H}^1$ and $v \in H$ by
\begin{align*}
(G_2 (v)u)[x] = v(x) u(x).
\end{align*}
% \andreas{\sout{It is known that both choices of $G$ are linear mappings from $H$ to $L_{\text{HS}} (Q^{1/2}(U);H)$, the space of Hilbert--Schmidt mappings, (see, e.g.},~\cite{K14} and \cite{LP17})\sout{, but we need to check that $G \in L(H;L(U;H))$.
% To this end, } }
To see that $G \in L(H;L(U;H))$ note that for $u,v \in H$, by the triangle inequality and Cauchy--Schwarz we have for $G_1$
\begin{align*}
\norm{G_1(v) u}{H} 
  &\le \sum_{i=1}^\infty |\langle v,e_i \rangle_H | |\langle u,e_{i} \rangle_H | 
  \le  \Bigl(\sum_{i=1}^\infty \langle v,e_i \rangle_H^2 \Bigr)^{1/2} 
	\Bigl(\sum_{i=1}^\infty \langle u,e_i \rangle_H ^2 \Bigr)^{1/2} = \norm{v}{H} \norm{u}{H}. 
  %\le \lambda_1^{-1/2} \norm{v}{\dot{H}^1} \norm{u}{H}
\end{align*}
Next, for $G_2$ with $v \in H$ and $u \in \dot{H}^1$, it holds that
\begin{align*}
	\norm{G_2(v) u}{H}^2 
	  &= \int^1_0 u(x)^2 v(x)^2 \, \dd x 
	  = \int^1_0 \Bigl( \sum^\infty_{i=1} \lambda_i^{1/2} \inpro{u}{e_i}{H} \lambda_i^{-1/2} e_i (x) \Bigr)^2 v(x)^2 \, \dd x \\ 
	&\le \Bigl( \sum^\infty_{i=1} \lambda_i |\inpro{u}{e_i}{H}|^2 \Bigr) \int^1_0 \Bigl( \sum^\infty_{i=1} \lambda_i^{-1} e_i(x)^2 \Bigr) v(x)^2 \, \dd x \\
	&\le \norm{u}{\dot{H}^1}^{2} \Bigl( 2 \sum^\infty_{i=1} \lambda_i^{-1} \Bigr) \int^1_0 v(x)^2 \, \dd x  
	=  \Bigl(2 \sum^\infty_{i=1} \lambda_i^{-1} \Bigr) \norm{u}{\dot{H}^1}^{2} \norm{v}{H}^{2}.
\end{align*}
Here, the first inequality is an application of the Cauchy--Schwarz inequality, while the second follows from the fact that the sequence ${(|e_i(x)|, i \in \N)}$ is bounded by $\sqrt{2}$ for all $x \in [0,1]$.
Therefore, we obtain
\begin{align*}		
	\norm{G_1}{L(H;L(H))} \le 1,
\qquad 
% \end{align*}
% and
% \begin{align*}		
\norm{G_2}{L(H;L(\dot{H}^1,H))} \le \Bigl( 2 \sum^\infty_{i=1} \lambda_i^{-1} \Bigr)^{1/2}.
\end{align*}
\end{example}

% In the following sections, we investigate mean-square stability properties of fully discrete (Galerkin) approximations based on time discretizations with rational approximations.

\subsection{Time discretization with rational approximations}\label{subsec:TimeRational}
 
Let us first recall that a \emph{rational approximation of order~$p$} of the exponential function is a rational function $R: \C \rightarrow \C$ satisfying that there exist constants $C,\delta > 0$ such that for all $z \in \C$ with $|z| < \delta$
\begin{align*}
 |R(z) - \exp(z)| \le C |z|^{p+1}.
\end{align*}
Since $R$ is rational there exist polynomials $r_n$ and $r_d$ such that $R = r_d^{-1} r_n$.
% for all $z \in \C$
% \begin{align*}
%  R(z) = \frac{r_\mathrm{n}(z)}{r_\mathrm{d}(z)}.
% \end{align*}
We want to consider rational approximations of the semigroup~$S$ generated by the operator~$-A$ and of its approximations~$-A_h$ as they were considered in~\cite{T06}.
With the introduced notation, the linear operator $R(\Delta t A_h)$ is given for all $v_h \in V_h$ by
\begin{align}
\label{eq:rational-linear-operator-approximation}
R(\Delta t A_h)v_h = r_\mathrm{d}^{-1}(\Delta t A_h) r_\mathrm{n}(\Delta t A_h)v_h = \sum_{k=1} ^{N_h} \frac{r_\mathrm{n}(-\Delta t \lambda_{h,k})}{r_\mathrm{d}(-\Delta t \lambda_{h,k})}\inpro{v_h}{e_{h,k}}{H}e_{h,k}.
\end{align}
% We consider two cases of discretizations of the stochastic integral in combination with the rational approximation: the operator~$\Dstoch{j}$ is first based on an Euler--Maruyama scheme and then on a Milstein scheme.

Let us start with the mean-square stability properties of a Galerkin Euler--Maruyama method, which is given by the recursion
\begin{align}\label{schemeEMbased}
%  \begin{split}
 X_h^{j+1} 
%   &= \left(R(\Delta t A_h)  + r_\mathrm{d}^{-1}(\Delta t A_h)( \Delta t P_h F + P_h G(\cdot)\Delta L^{j} )\right)X_h^j, \\
  &= (\Ddet + \Dpart{EM}{j}) X_h^j
%  X_h^0 
%   &= P_h X_0 
%  \end{split}
\end{align}
 for $j \in\N_0$ with initial condition $X_h^0 = P_h X_0$, where 
% The corresponding operators from the fully discrete scheme~\eqref{FullDiscretScheme} are then given by
\begin{align}\label{eq:abstractLinOpRational}
  \Ddet 
    = R(\Delta t A_h)+ r_\mathrm{d}^{-1}(\Delta t A_h) \Delta t P_hF,
  \quad %\text{ and } \quad 
  \Dstoch{j} 
    = \Dpart{EM}{j} 
    = r_\mathrm{d}^{-1}(\Delta t A_h) P_h G(\cdot) \Delta L^{j}
\end{align}
with $\Delta L^{j} = L(t_{j+1}) - L(t_j)$.
Note that the linear operators $(\Dstoch{j}, j \in \N_0)$ satisfy all assumptions of Corollary~\ref{cor:generalApproxMSstabiid} since they only depend on the \levy increments $(\Delta L^j,j\in\N_0)$. For this type of numerical approximation, the result from Corollary~\ref{cor:generalApproxMSstabiid} can be specified:

\begin{proposition}
\label{Proposition:MSStableFD}
The zero solution of the numerical method~\eqref{schemeEMbased} is asymptotically mean-square stable if and only if
\begin{align*}
\cS = \Ddet \otimes \Ddet + \Delta t \, (C \otimes C) q \in L(V^{(2)}_h)
\end{align*}
satisfies that $\rho(\cS)<1$, where $q = \sum_{k=1}^{\infty} \mu_k f_k \otimes f_k \in U^{(2)}$ and
% $\Ddet$ is given in~\eqref{eq:abstractLinOpRational}, 
$C \in L(U; L(V_h))$  with
\begin{align*}
 Cu = r_\mathrm{d}^{-1}(\Delta t A_h) P_h G(\cdot) u.
\end{align*}
\end{proposition}

\begin{proof}
Note that since $V_h$ is finite-dimensional, $L(V_h) = L_{\text{HS}} (V_h)$ so $(C \otimes C)$ is well-defined as an element of $L(U^{(2)}, L^{(2)}_{\text{HS}} (V_h)) \subset L(U^{(2)}, L(V^{(2)}_h))$ by \cite[Lemma 3.1(ii)]{KLL17},
which yields for $j\in\N$
\begin{align*}
\E[ &\Dpart{EM}{j} \otimes \Dpart{EM}{j} ] = \E [ C \Delta L^{j} \otimes C \Delta L^{j}] = ( C \otimes C ) \E [\Delta L^{j} \otimes \Delta L^{j}].
\end{align*}
Since $\E [\Delta L^{j} \otimes \Delta L^{j}] = \Delta t \, q$ by Lemma~\ref{lem:appendix:em}, 
the proof is completed with Corollary~\ref{cor:generalApproxMSstabiid}.
\end{proof}

The still rather abstract condition can be specified to an explicit sufficient condition.

\begin{corollary}\label{Cor:sufficientConditionGeneralRatApprox}
A sufficient condition for asymptotic mean-square stability of~\eqref{schemeEMbased} is
\begin{align*}
    &\Bigl( \max_{k = 1,\dots,N_h} |R(-\Delta t \lambda_{h,k})| + \max_{k = 1,\dots,N_h} |r_\mathrm{d}^{-1}(-\Delta t \lambda_{h,k})| \Delta t \norm{ F}{L(H)} \Bigr)^2 \\
      &\hspace{5cm}+   \max_{k = 1,\dots,N_h} |r_\mathrm{d}^{-1}(-\Delta t \lambda_{h,k})|^2 \Delta t \trace(Q) \norm{G}{L(H;L(U;H))}^2 < 1.
\end{align*}
\end{corollary}

\begin{proof}
We first note that by the triangle inequality and the properties of the linear operator induced by the rational approximation $R$ defined in Equation~\eqref{eq:rational-linear-operator-approximation} 
we obtain that
\begin{align*}
\norm{\Ddet}{L(V_h)} &= \norm{R(\Delta t A_h)+ r_\mathrm{d}^{-1}(\Delta t A_h) \Delta t P_hF}{L(V_h)} \\
&\le\max_{k=1,\dots,N_h} |R(-\Delta t \lambda_{h,k})| + \max_{k = 1,\dots,N_h} |r_\mathrm{d}^{-1}(-\Delta t \lambda_{h,k})| \Delta t \norm{F}{L(H)}
\end{align*}
and, similarly
\begin{align*}
\norm{C}{L(U;L(V_h))} \le \norm{r_\mathrm{d}^{-1}(\Delta t A_h ) }{L(V_h)} \norm{G}{L(H;L(U;H))} \le \max_{k=1,\dots,N_h} |r_\mathrm{d}^{-1}(-\Delta t \lambda_{h,k})| \norm{G}{L(H;L(U;H))}.
\end{align*}
Since 

\begin{align*}
 \norm{(C \otimes C) q}{L(V_h^{(2)})}
%   = \sum_{k=1}^{\infty} \mu_k \norm{(C f_k \otimes C f_k)}{L(V_h^{(2)})}
  \le \sum_{k=1}^{\infty} \mu_k \norm{C f_k }{L(V_h)}^2
  \le \trace(Q) \norm{C}{L(U;L(V_h))}^2
\end{align*}
and $ \norm{\Ddet \otimes \Ddet}{L(V^{(2)}_h)} = \norm{\Ddet}{L(V_h)}^2$,
% \begin{align*}
% \norm{\cS}{L(V^{(2)}_h)} &\le \norm{\Ddet \otimes \Ddet}{L(V^{(2)}_h)} + \Delta t \norm{(C \otimes C) q}{L(V_h^{(2)})} \\
% &\le \norm{\Ddet \otimes \Ddet}{L(V^{(2)}_h)} + \Delta t \sum_{k=1}^{\infty} \mu_k \norm{(C f_k \otimes C f_k)}{L(V_h^{(2)})} \\
% &\andreas{\le \norm{\Ddet \otimes \Ddet}{L(V^{(2)}_h)} + \Delta t \sum_{k=1}^{\infty} \mu_k \norm{C f_k }{L(V_h)}^2} \\
% &\le \norm{\Ddet}{L(V_h)}^2 + \Delta t \trace(Q) \norm{C}{L(U;\andreas{L(V_h)})}^2,
% \end{align*}
% \andreas{\sout{where we used $\norm{\cdot}{L(V_h\otimes V_h)} = \norm{\cdot}{L(V_h)}^2$,}} 
we obtain the claimed condition, which is sufficient by Corollary~\ref{cor:generalApproxMSstabiid}.
\end{proof}

We continue with the higher order Milstein scheme. Applying~\cite{BL12} in our context reads
\begin{align}\label{Eq:MilsteinScheme}
\begin{split}
X_h^{j+1} 
  = (\Ddet + \Dpart{EM}{j} + \Dpart{M}{j})X_h^j,
%   \big( R(\Delta t A_h)  &+ r_\mathrm{d}^{-1}(\Delta t A_h) (\Delta t P_h F + P_h G(\cdot) \Delta L^{j})\big)X_h^j \\
%     &+  \int_{t_{j}}^{t_{j+1}} r_\mathrm{d}^{-1}(\Delta t A_h) P_h G\Bigl(\int_{t_{j}}^s G(X_h^j)\, \dd L(r)\Bigr)\, \dd L(s).
\end{split}
\end{align}
where $\Ddet$ and $\Dpart{EM}{j}$ are as in~\eqref{eq:abstractLinOpRational} and
\begin{align*}
% \Dpart{EM}{j} 
% &= r_\mathrm{d}^{-1}(\Delta t A_h) P_h G(\cdot) \Delta L^{j}, \\
\Dpart{M}{j} 
&=  \sum_{k,\ell = 1}^\infty r_\mathrm{d}^{-1}(\Delta t A_h) \sqrt{\mu_k \mu_\ell}P_h G(G(\cdot )f_k)f_\ell \int_{t_{j}}^{t_{j+1}} \int_{t_{j}}^s \dd L_k(r) \, \dd L_\ell(s).
\end{align*}

% The iterated stochastic integrals can be represented with the \KL expansion~\eqref{eq:karhunenloeveexpansion} of the L\'evy process by
% \begin{align*}
% \int_{t_{j}}^{t_{j+1}} &r_\mathrm{d}^{-1}(\Delta t A_h) P_h G\Bigl(\int_{t_{j}}^s G(X_h^j)\, \dd L(r)\Bigr) \, \dd L(s) \\
%   &\hspace{3cm}= \sum_{k,\ell = 1}^\infty \sqrt{\mu_k \mu_\ell} 
% 	  r_\mathrm{d}^{-1}(\Delta t A_h) P_h G(G(X_h^j)f_k)f_\ell 
% 	  \int_{t_{j}}^{t_{j+1}} \int_{t_{j}}^s \dd L_k(r) \, \dd L_\ell(s).
% \end{align*}
% Thus, the stochastic operator~$\Dstoch{j}$ can be written as the sum of the stochastic operator of the Euler--Maruyama scheme~\eqref{eq:abstractLinOpRational} and an operator including the iterated stochastic integrals, i.e.,
% \begin{align}\label{Eq:DstochMilstein}
% \Dstoch{j} = \Dpart{EM}{j} + \Dpart{M}{j},
% \end{align} 
% where
% \begin{align*}
% \Dpart{EM}{j} 
% &= r_\mathrm{d}^{-1}(\Delta t A_h) P_h G(\cdot) \Delta L^{j}, \\
% \Dpart{M}{j} 
% &=  \sum_{k,\ell = 1}^\infty r_\mathrm{d}^{-1}(\Delta t A_h) \sqrt{\mu_k \mu_\ell}P_h G(G(\cdot )f_k)f_\ell \int_{t_{j}}^{t_{j+1}} \int_{t_{j}}^s \dd L_k(r) \, \dd L_\ell(s).
% \end{align*}

\begin{remark}
	\label{rem:g_commutativity}
In order to compute the iterated integrals of $\Dpart{M}{j}$, one may assume (cf.\ \cite{BL12,JR15}) that for all $H$-valued, adapted stochastic processes $\chi= (\chi(t),t\geq 0)$ and all $i,j\in \mathbb{N}$, the diffusion operator $G$ satisfies the commutativity condition
	\begin{align*}
	G(G(\chi)f_j ) f_i = G(G(\chi)f_i)f_j.
	\end{align*}
Under this assumption satisfied in Example~\ref{Ex:stochastic_heat_equation}, $\Dpart{M}{j}$ simplifies to
\begin{align*}
\Dpart{M}{j}
&= \frac{1}{2}\sum_{k,\ell = 1}^\infty \sqrt{\mu_k \mu_\ell} r_\mathrm{d}^{-1}(\Delta t A_h) P_h G(G(X_h^j)f_k)f_\ell (\Delta L_{k}^{j} \Delta L_{\ell}^{j} -  \Delta[L_k,L_\ell]^j),
\end{align*}
where 
$\Delta[L_k,L_\ell]^j = [L_k,L_\ell]_{t_{j+1}} - [L_k,L_\ell]_{t_{j}}$. Here, $[L_k,L_\ell]_{t}$ denotes the 
quadratic covariation of $L_k$ and $L_\ell$ evaluated at $t \ge 0$, which is straightforward to compute when $L_k, L_\ell$ are jump-diffusion processes (cf.\ \cite{BL12}). For the simulation of more general \levy processes in the context of SPDE approximation, we refer to \cite{DHP12,BS16}.
% \andreasJKU{\att{Note that Andrea Barth and Andreas Stein have an arXiv-preprint on the simulation of \levy processes (arXiv:1612.05541).
% 
% I decided to remove further discussions on simulating \levy areas in the Wiener noise case.}}
% 
%\andreas{\sout{If the diffusion operator does not satisfy this commutativity condition,
%	the simulation of the iterated integrals by the increments of the \levy processes $(L_i,i\in\N)$ is not possible. For Wiener processes, a solution exists (see, e.g.,~\cite{W01}) and mean-square stability analysis for finite-dimensional SDEs can be done with appropriately truncated stochastic \levy areas}~\cite{BS12}. \att{Not used. Remove? \annika{Andreas T should decide if space admits.}}}
\end{remark}

As for the Euler--Maruyama scheme, Corollary~\ref{cor:generalApproxMSstabiid} can be specified for this Milstein scheme.

\begin{proposition}\label{Proposition:MilsteinRationalApproxS}
Assume that the bilinear mapping $C'(u_1,u_2) = r_\mathrm{d}^{-1}(\Delta t A_h) P_h G(G(\cdot)u_1)u_2$ for $u_1, u_2 \in U$ can be uniquely extended to a mapping $C' \in L(U^{(2)},L(V_h))$. Then the zero solution of~\eqref{Eq:MilsteinScheme} is asymptotically mean-square stable if and only if
\begin{align*}
 \cS = \Ddet \otimes \Ddet + \Delta t \, (C \otimes C) q +\frac{\Delta t^2}{2} (C' \otimes C') q'
\end{align*}
satisfies that $\rho(\cS) < 1$. 
Here, $q' = \sum_{k,\ell=1}^{\infty} \mu_k \mu_\ell(f_k \otimes f_{\ell}) \otimes (f_k \otimes f_{\ell}) \in U^{(4)}$ and $C$ and $q$ as in Proposition~\ref{Proposition:MSStableFD}.
\end{proposition}  

\begin{proof}
Note that $C' \otimes C' : U^{(4)} \to L(V^{(2)}_h)$ and $C' \otimes C : U^{(2)} \otimes U \to L(V^{(2)}_h)$ are well-defined by the same arguments as in Proposition~\ref{Proposition:MSStableFD}. Since $\Dstoch{j} = \Dpart{EM}{j} + \Dpart{M}{j}$, we obtain for $j\in\mathbb{N}_0$
\begin{align*}
 \E[\Dstoch{j} \otimes \Dstoch{j}] 
  &= \E[\Dpart{EM}{j} \otimes \Dpart{EM}{j}] + \E[\Dpart{M}{j} \otimes \Dpart{EM}{j}] \\
    &\quad + \E[\Dpart{EM}{j} \otimes \Dpart{M}{j}] + \E[\Dpart{M}{j} \otimes \Dpart{M}{j}].
\end{align*}
The first term and $\Ddet\otimes \Ddet$ are given in Proposition~\ref{Proposition:MSStableFD}.
% , observe that the underlying operator $\Dpart{EM}{j}$ coincides with the operator $\Dstoch{j}$ from Equation~\eqref{eq:abstractLinOpRational} and therefore, the first two components of $\cS$ follow from the proof of Proposition~\ref{Proposition:MSStableFD}. 
We conclude for the second term with Lemma~\ref{lem:appendix:mil},
writing $\Delta^{(2)} L = \sum_{k,\ell=1}^\infty \sqrt{\mu_k \mu_\ell} \left( \int_{t_{j}}^{t_{j+1}} \int_{t_{j}}^s \dd L_k(r) \, \dd L_\ell(s)  \right) f_k \otimes f_{\ell} $,
\begin{align*}
\E[\Dpart{M}{j} \otimes \Dpart{EM}{j}] 
  &= \E\bigl[C' \Delta^{(2)} L^j \otimes C \Delta L^j \bigr] 
  = (C' \otimes C) \E\bigl[\Delta^{(2)} L^j \otimes \Delta L^j\bigr]
  = 0
\end{align*}
% and by Lemma~\ref{lem:appendix:mil} $\E\left[\Delta^{(2)} L^j \otimes \Delta L^j\right] = 0$.
and analogously the same for the third term.
Finally, Lemma~\ref{lem:appendix:mil} yields
\begin{align*}
 \E[\Dpart{M}{j} \otimes \Dpart{M}{j}] 
  = \E\bigl[C' \Delta^{(2)} L^j \otimes C' \Delta^{(2)} L^j \bigr] 
  = (C' \otimes C') \E\bigl[\Delta^{(2)} L^j \otimes \Delta^{(2)} L^j \bigr]
\end{align*}
% and by Lemma~\ref{lem:appendix:mil}
% \begin{align*}
% E\left[\Delta^{(2)} L^j \otimes \Delta^{(2)} L^j \right] = \frac{\Delta t^2}{2} q'.
% \end{align*}
and the statement follows directly from Corollary~\ref{cor:generalApproxMSstabiid}.
\end{proof}

% Before we apply the derived conditions to examples of rational approximations, we finish this part with a remark on the regularity assumption on the Milstein term.

\begin{remark}
The assumption on~$C'$ in Proposition~\ref{Proposition:MilsteinRationalApproxS}
% \andreas{\sout{, that the bilinear mapping~$C'$ can be uniquely extended to a mapping in the space $L(U^{(2)};L(V_h))$,}}
holds for the operators $G_1$ and $G_2$ in the setting of Example~\ref{Ex:stochastic_heat_equation}. 
One can  get rid of this assumption by using that the bound on $G \in L(H;L(U;H))$ allows for an extension of the bilinear mapping to the \emph{projective tensor product space} $U \otimes_\pi U$, cf.~\cite{KLL17}. One would then have to assume additional regularity on $L$ to ensure that $\Delta^{(2)} L^j$ in the proof of Proposition~\ref{Proposition:MilsteinRationalApproxS} is in the space $L^2(\Omega;U \otimes_\pi U)$. Alternatively, one considers finite-dimensional \emph{truncated} noise, which leads to equivalent norms.
\end{remark}

\subsection{Examples of rational approximations}
\label{Subsec:ExamplesRationalApproximations}

Let us next consider specific choices of rational approximations~$R$ and investigate their influence on mean-square stability.
First, we derive sufficient conditions based on Corollary~\ref{Cor:sufficientConditionGeneralRatApprox} for Euler--Maruyama schemes with standard rational approximations. More specifically, we consider the backward Euler, the Crank--Nicolson, and the forward Euler scheme.
\begin{theorem}
\label{Thm:MSStableFD}
Consider the approximation scheme~\eqref{schemeEMbased}.
\begin{enumerate}
\item\label{prop-item:backwardEuler} \emph{(Backward Euler scheme)} Let $R$ be given by
% \begin{align*}%\label{eq:RatApproxBE}
$R(z) = (1-z)^{-1}$.
% \end{align*}
Then \eqref{schemeEMbased} is asymptotically mean-square stable if
\begin{align*}
%\label{eq:suffcondition_euler}
\frac{(1 + \Delta t \norm{F}{L(H)})^2 + \Delta t \trace(Q) \norm{G}{L(H;L(U;H))}^2}{(1 + \Delta t \lambda_{h,1})^2} < 1.
\end{align*}
\item%\label{prop-item:CrankNicolson}
 \emph{(Crank--Nicolson scheme)} Let $R$ be given by
% \begin{align*}%\label{eq:RatApproxCN}
$R(z) = (1 + z/2)/(1-z/2)$.
% \end{align*}
Then \eqref{schemeEMbased} is asymptotically mean-square stable if
\begin{align*}%\label{Eq:suffCondCN}
\left(\max_{k \in \{ 1,N_h\}} \bigg| \frac{1- \Delta t \lambda_{h,k}/2}{1 + \Delta t \lambda_{h,k}/2} \bigg| + \Delta t \frac{\| F\|_{L(H)}}{(1 + \Delta t \lambda_{h,1}/2)}\right)^2 + \Delta t \frac{ \trace(Q) \| G \|_{L(H;L(U;H))}^2}{(1 + \Delta t \lambda_{h,1}/2)^2} < 1.
\end{align*}
\item%\label{prop-item:forwardEuler}
 \emph{(Forward Euler scheme)} Let $R$ be given by
% \begin{align*}%\label{eq:RatApproxFE}
$R(z) = 1+z$.
% \end{align*}
Then \eqref{schemeEMbased} is asymptotically mean-square stable if
\begin{align*}% \label{Eq:suffcondition_euler}
\Bigl( \max_{\ell \in \{1,N_h\}} |1 - \Delta t \lambda_{h,\ell}| + \Delta t \norm{F}{L(H)} \Bigr)^2 + \Delta t \trace(Q) \norm{G}{L(H;L(U;H))}^2 < 1.
\end{align*}
\end{enumerate}
\end{theorem}

\begin{proof}
Let us start with the backward Euler scheme. Since the functions $r_\mathrm{d}^{-1}(z)$ and $R(z)$ are equal and it holds for all $k = 1, \dots,N_h$ that $|R(-\Delta t \lambda_{h,k} )| \le |R(-\Delta t \lambda_{h,1})|$, we obtain by Corollary~\ref{Cor:sufficientConditionGeneralRatApprox} asymptotic mean-square stability if
\begin{align*}
(1 + \Delta t \lambda_{h,1})^{-2} \left( (1 + \Delta t \norm{F}{L(H)})^2 + \Delta t \trace(Q) \norm{G}{L(H;L(U;H))}^2 \right) < 1.
\end{align*}

For the Crank--Nicolson scheme, note that $R$ is decreasing on $\R^-$ and that $R(z) \in [-1,1]$ for all $z\in\R^-$. Thus, the maximizing eigenvalue is either the largest, $\lambda_{h,N_h}$, or the smallest, $\lambda_{h,1}$, and therefore, 
\begin{align*}
| R(- \Delta t \lambda_{h,k})| \leq \max_{\ell\in\{1,N_h\}} | R (- \Delta t \lambda_{h,\ell})|.
\end{align*} 
Since $|r_\mathrm{d}^{-1}(- \Delta t \lambda_{h,k})| \leq |r_\mathrm{d}^{-1}(- \Delta t \lambda_{h,1})|$ for all $k = 1,\dots,N_h$, the claim follows with Corollary~\ref{Cor:sufficientConditionGeneralRatApprox}.

By the same arguments, 
we obtain for the forward Euler scheme that $|R(-\Delta t \lambda_{h,i})|$ is maximized either at $z=-\Delta t\lambda_{h,1}$ or $z=-\Delta t\lambda_{h,N_h}$. Therefore, since $r_\mathrm{d}^{-1} (z) = 1$, the claim follows again with Corollary~\ref{Cor:sufficientConditionGeneralRatApprox},
which finishes the proof.
\end{proof}

For the Milstein scheme, Proposition~\ref{Proposition:MilsteinRationalApproxS} yields the following sufficient condition.
\begin{proposition}\label{Prop:MSStableMilstein}
Under the assumptions of Proposition~\ref{Proposition:MilsteinRationalApproxS}, the Milstein scheme~\eqref{Eq:MilsteinScheme} with $R(z) = (1-z)^{-1}$ is asymptotically mean-square stable if
\begin{align*}%\label{eq:estimateMilsteinbased}\begin{split}
(1 + \Delta t \norm{F}{L(H)})^2 + \Delta t \trace (Q) \norm{G}{L(H;L(U;H))}^2 
  +  \frac{\Delta t^2}{2} \trace (Q)^2 &\norm{G}{L(H;L(U;H))}^4 \\
  &\hspace{1.5cm}< (1+\Delta t \lambda_{h,1})^2.%\end{split}
\end{align*}
\end{proposition}

\begin{proof}
In the same way as in the proof of Corollary~\ref{Cor:sufficientConditionGeneralRatApprox}, we bound 
\begin{align*}
 \| (C' \otimes C') q' \|_{L(V^{(2)}_h)} \le \|C' \|^2_{L(U^{(2)};L(V_h))} \trace(Q)^2 \le (1+\Delta t \lambda_{h,1})^{-2} \norm{G}{L(H;L(U;H))}^4 \trace(Q)^2.
\end{align*}
Hence, our assumption ensures that $\|\cS \|_{L(V_h^{(2)})} <1$, which by Corollary~\ref{cor:generalApproxMSstabiid} is sufficient for asymptotic mean-square stability.
\end{proof}
Note that the sufficient condition for the Milstein scheme is more restrictive than the sufficient condition presented in Theorem~\ref{Thm:MSStableFD}(\ref{prop-item:backwardEuler}) for the backward Euler--Maruyama method due to the additional positive term in the estimate in Proposition~\ref{Prop:MSStableMilstein}. 

\subsection{Relation to the mild solution}\label{subsec:parallel_stab_solution_approx}

To connect existing results on asymptotic mean-square stability of~\eqref{eq:SPDE} to the results for discrete schemes in Section~\ref{Subsec:ExamplesRationalApproximations}, 
we have to restrict ourselves to $Q$-Wiener processes $W = (W(t),t\geq 0)$ due to the framework for analytical solutions in~\cite{L06}. Specifically, we consider
\begin{align}\label{eq:SPDEWiener}
\dd X(t) = (AX(t) + F X(t)) \, \dd t + G(X(t)) \, \dd W(t).
\end{align}
%%%%%%%%%%%%%%%%%%%%
% \annika{I do not see that this is necessary information. Put later where required}
%%%%%%%%%%%%%%%%%%%%
% Note that for $Q$-Wiener processes, the \KL expansion~\eqref{eq:karhunenloeveexpansion} becomes
% \begin{align*}
% W(t) = \sum_{i=1}^\infty \sqrt{\mu_i} \beta_i(t) f_i,
% \end{align*}
% where $(\beta_i,i\in \N)$ is a sequence of independent, real-valued Brownian motions.

The following special case of \cite[Proposition 3.1.1]{L06} provides a sufficient condition for the asymptotic mean-square stability of~\eqref{eq:SPDE} by a Lyapunov functional approach.
\begin{theorem}\label{thm:ms_stability_Lyapunov}
Assume that $X_0=x_0 \in \dot{H}^1$ is deterministic and there exists $c>0$ such that
\begin{align*}
2\langle v,A v + F(v)\rangle_H  + \trace[G(v)Q(G(v))^* ] \leq -c\|v\|^2_H
\end{align*}
 for all $v\in \dot{H}^2$. Then the zero solution of~\eqref{eq:SPDEWiener} is asymptotically mean-square stable.
\end{theorem}

We use this theorem to derive simultaneous sufficient mean-square stability conditions for~\eqref{eq:SPDEWiener} and the corresponding backward Euler scheme~\eqref{schemeEMbased}.

\begin{corollary}\label{cor:stability_connection}
% Consider the backward Euler scheme in the framework of Theorem~\ref{Thm:MSStableFD}(\ref{prop-item:backwardEuler}) and 
Assume that $X_0=x_0 \in \dot{H}^1$ is deterministic. Then
the zero solutions of~\eqref{eq:SPDEWiener} and its approximation~\eqref{schemeEMbased} with $R(z) = (1-z)^{-1}$ are asymptotically mean-square stable for all $h$ and $\Delta t$ if
\begin{equation}
\label{eq:simultaneous_stability}
2\left( \norm{F}{L(H)}- \lambda_1 \right) + \trace(Q) \norm{G}{L(H;L(U;H))}^2 < 0.
\end{equation}
% is a sufficient condition for the asymptotic mean-square stability of the zero solutions of both, \eqref{eq:SPDEWiener} and its approximation~\eqref{schemeEMbased}, independent of the values of $h$ and $\Delta t$.
\end{corollary}

\begin{proof}
% 	Let us assume that 
% 	\begin{align*}
	%\label{eq:simultaneous_stability}
%  2\left( \norm{F}{L(H)}- \lambda_1 \right) + \trace(Q) \norm{G}{L(H;L(U;H))}^2 < 0.
% 	\end{align*}
	We show first that \eqref{eq:simultaneous_stability} yields asymptotic mean-square stability of~\eqref{eq:SPDEWiener} by reducing it to the assumption in Theorem~\ref{thm:ms_stability_Lyapunov}.
% 	which by Theorem~\ref{thm:ms_stability_Lyapunov} follows from
% 	\begin{align*}
% 	2\langle v,A v + F(v)\rangle_H  + \trace[G(v)Q(G(v))^* ] < -c\|v\|^2_H
% 	\end{align*}
% 	for some $c > 0$.
	For the second term there, note that for any $v \in \dot{H}^2$, 
\begin{align*}
 \trace[G(v)Q(G(v))^*]
    &=\trace[(G(v))^*G(v)Q] 
      = \sum_{k=1}^\infty \langle G(v)Q f_k , G(v)f_k \rangle \\ 
    &\leq \sum_{k=1}^\infty \mu_k \|G\|_{L(H;L(U;H))}^2 \|v\|^2_H \| f_k \|_U^2 
      = \trace(Q) \|G\|_{L(H;L(U;H))}^2   \|v\|_H^2,
\end{align*}
where the first equality follows from the properties of the trace.
% fact that the trace operator is invariant under cyclic permutations.
The first term satisfies
\begin{align*}
 \langle v,A v + F(v)\rangle 
  = \langle v,F(v)\rangle + \langle v,A v\rangle 
  \le \norm{F}{L(H)} \|v\|^2_H - \|v\|^2_1 
  \le ( \norm{F}{L(H)} -\lambda_1) \|v\|_H^2
\end{align*}
% where the last inequality follows from 
using the definition of $\|\cdot\|_1$.
Altogether, we therefore obtain
\begin{align*}
2\langle v,A v + F(v)\rangle_H  &+ \trace[G(v)Q(G(v))^*] \\ &\le \bigl( 2 \left(\norm{F}{L(H)} - \lambda_1 \right)+ \trace(Q) \norm{G}{L(H;L(U;H))}^2 \bigr) \norm{v}{H}^2,
\end{align*}
 i.e., with~\eqref{eq:simultaneous_stability} asymptotic mean-square stability of~\eqref{eq:SPDEWiener} by setting
\begin{align*}
	c = - \bigl( 2 \left(\norm{F}{L(H)} - \lambda_1 \right)+ \trace(Q) \norm{G}{L(H;L(U;H))}^2 \bigr).
\end{align*}

We continue with~\eqref{schemeEMbased} observing first that $\lambda_{h,1} \ge \lambda_1$ by~\eqref{eq:eigenvalue_inequality}. Therefore, the condition
% , the asymptotic mean-square stability then follows if we can show that
% \begin{align*}
% (1 + \Delta t \norm{F}{L(H)})^2 + \Delta t \trace(Q) \norm{G}{L(H;L(U;H))}^2 < (1 + \Delta t \, \lambda_1)^2.
% \end{align*} 
in Theorem~\ref{Thm:MSStableFD}(\ref{prop-item:backwardEuler}) yields 
\begin{align*}
\Delta t \bigl( 2 \left( \norm{F}{L(H)} - \lambda_1 \right) + \trace(Q) \norm{G}{L(H;L(U;H))}^2 \bigr) + {\Delta t}^2 \bigl( \norm{F}{L(H)}^2 - \lambda_1^2 \bigr) < 0.
\end{align*}
This is satisfied and finishes the proof since the first term is negative by assumption and so is the second using \eqref{eq:simultaneous_stability} and
\begin{align*}
\norm{F}{L(H)}^2 - \lambda_1^2
 & = \bigl( \norm{F}{L(H)} + \lambda_1 \bigr) \bigl( \norm{F}{L(H)} - \lambda_1 \bigr) \\
 & \le \left( \norm{F}{L(H)} + \lambda_1 \right) \bigl( \left( \norm{F}{L(H)}- \lambda_1 \right) + 2^{-1} \trace(Q) \norm{G}{L(H;L(U;H))}^2 \bigr) 
 < 0.
 \qedhere
\end{align*}
% This finishes the proof.
\end{proof}

Note that under \eqref{eq:simultaneous_stability} in Corollary~\ref{cor:stability_connection}, the backward Euler--Maruyama scheme preserves the qualitative behaviour of the analytical solution without any restriction on~$h$ and~$\Delta t$. Hence, it can be applied to numerical methods requiring different refinement parameters in parallel such as multilevel Monte Carlo, which efficiently approximate quantities $\mathbb{E}[\varphi(X(T))]$ (see, e.g.,~\cite{BLS13,BL12a} for details). Here, it is essential that the behaviour is preserved on all refinement levels~\cite{A13}.

% \todo{[MOVED HERE SINCE BEFORE NOT DISCUSSED WHEN SOLUTION AMMS, to be integrated]}
% \begin{remark}%\label{remark:homogeneousBE}
On the other hand, note that in the homogeneous case, i.e., $F=0$, the stability condition in Theorem~\ref{Thm:MSStableFD}(\ref{prop-item:backwardEuler}) reduces to
\begin{align*}
\trace(Q) \norm{G}{L(H;L(U;H))}^2 < \lambda_{h,1} (2 + \Delta t \lambda_{h,1})
\end{align*}
so that even if \eqref{eq:SPDE} is asymptotically mean-square unstable, its approximation~\eqref{schemeEMbased} can always be rendered stable by letting $\Delta t$ be large enough. In that case the analytical solution and its approximation have a different qualitative behaviour for large times.
% \end{remark}

\begin{remark}
Based on Theorem~\ref{Thm:MSStableFD}, it is also possible to examine the relation between asymptotic mean-square stability of~\eqref{eq:SPDEWiener} and its approximation by the other rational approximations. However, due to the nature of the sufficient conditions in Theorem~\ref{Thm:MSStableFD}, analogous results to Corollary~\ref{cor:stability_connection} include restrictions on~$h$ and~$\Delta t$.  
\end{remark}

For the Milstein scheme considered in Proposition~\ref{Prop:MSStableMilstein} we can also derive a sufficient condition for the simultaneous mean-square stability not relying on~$h$ and~$\Delta t$. However, due to the additional term in Proposition~\ref{Prop:MSStableMilstein}, the condition becomes slightly more restrictive than in Corollary~\ref{cor:stability_connection}. More precisely we obtain the following:
\begin{corollary}
% Consider the backward Euler--Milstein scheme in the framework of Proposition~\ref{Prop:MSStableMilstein} and assume further that $X_0 = x_0 \in \dot{H}^1$ is deterministic and $F = 0$. Then the inequality
Assume that $X_0=x_0 \in \dot{H}^1$ is deterministic and $F = 0$. Then
the zero solutions of~\eqref{eq:SPDEWiener} and its Milstein approximation~\eqref{Eq:MilsteinScheme} with $R(z) = (1-z)^{-1}$ are asymptotically mean-square stable for all $h$ and $\Delta t$ if
\begin{align*}%\label{eq:simultaneous_stability_milstein}
-\sqrt{2} \lambda_1 + \operatorname{Tr}(Q) \| G \|^2_{L(H;L(U;H))}< 0.
\end{align*}
% is a sufficient condition for the asymptotic mean-square stability of the zero solutions of both, \eqref{eq:SPDEWiener} and its approximation~\eqref{Eq:MilsteinScheme}, independent of $h$ and $\Delta t$.
\end{corollary}

\begin{proof}
% Let us assume that $- \sqrt{2} \lambda_1 + \operatorname{Tr}(Q) \| G \|^2_{L(H;L(U;H))}< 0$.
The asymptotic mean-square stability of~\eqref{eq:SPDEWiener} follows by Corollary~\ref{cor:stability_connection} since
\begin{align*}
-2 \lambda_1 +\operatorname{Tr}(Q) \| G \|^2_{L(H;L(U;H))} < - \sqrt{2} \lambda_1 + \operatorname{Tr}(Q) \| G \|^2_{L(H;L(U;H))}< 0.
\end{align*}
% we get the asymptotic mean-square stability of the zero solution of~\eqref{eq:SPDEWiener} by Corollary~\ref{cor:stability_connection}. 
The sufficient condition for~\eqref{Eq:MilsteinScheme} in Proposition~\ref{Prop:MSStableMilstein} can be rewritten as
\begin{align*}%\label{eq:sufficientCond_Milstein_corollary}
\Delta t ( - 2 \lambda_{h,1} + \operatorname{Tr}(Q) \| G \|^2_{L(H;L(U;H))} ) + \Delta t^2( -2 \lambda_{h,1}^2  + \operatorname{Tr}(Q)^2\| G \|^4_{L(H;L(U;H))}) < 0.
\end{align*}
The first summand is negative since
\begin{align*}
- 2 \lambda_{h,1} + \operatorname{Tr}(Q) \| G \|^2_{L(H;L(U;H))} < -\sqrt{2}\lambda_1 + \operatorname{Tr}(Q) \| G \|^2_{L(H;L(U;H))} < 0.
\end{align*}
The assumption $\sqrt{2} \lambda_1 > \operatorname{Tr}(Q) \| G \|^2_{L(H;L(U;H))}$ implies for the second summand that
\begin{align*}
-2 \lambda_{h,1}^2  + \operatorname{Tr}(Q)^2\| G \|^4_{L(H;L(U;H))} < 0.
\end{align*}
Thus, asymptotic mean-square stability of~\eqref{Eq:MilsteinScheme} follows.
\end{proof}

\section{Simulations}\label{sec:numerics} 

In this section we adopt the setting of Example~\ref{Ex:stochastic_heat_equation} and use numerical simulations to illustrate our theoretical results. More specifically, we consider the stochastic heat equation
\begin{align}\label{eq:StochHeat}
\dd X(t) = \nu \Delta X(t) \, \dd t + G(X(t)) \, \dd W(t).
\end{align}
with $X_0(x) = \sqrt{30} x(1-x)$, then $\E [\| X_0 \|_H^2] = 1$.
% The eigenvalues $(\mu_i,i\in\N)$ of the operator $Q$ obtained by the relation $Q e_i = \mu_i e_i$ are assumed to be $\mu_i = C_\mu i^{-\alpha}$, where $C_\mu > 0$ and $\alpha > 1$. 
We consider a $Q$-Wiener process
% as driving noise with respect to the eigenbasis $(e_i, i \in \N)$, i.e.,
% \begin{align*}
$W(t) = \sum_{i=1}^\infty \sqrt{\mu_i} \beta_i(t) e_i$,
% \end{align*}
where $(\beta_i,i\in \N)$ is a sequence of independent, real-valued Brownian motions, and assume $\mu_i = C_\mu i^{-\alpha}$ with $C_\mu > 0$ and $\alpha > 1$. Here, $C_\mu$ scales the noise intensity and $\alpha$ controls the space regularity of~$W$, see, e.g., \cite{LPS14,LS15}.

\subsection{Spectral Galerkin methods}
\label{subsec:spectral_galerkin}

% We examine spectral Galerkin methods for the stochastic heat equation with diffusion operator $G_1$ of Example~\ref{Ex:stochastic_heat_equation} first. 
For $G=G_1$ in Example~\ref{Ex:stochastic_heat_equation}, we obtain with the approach presented in \cite[Section 6.4]{K14} the infinite-dimensional counterpart of the geometric Brownian motion
\begin{align*}
X(t) = \sum_{i=1}^\infty \langle X(t),e_i \rangle_H e_i = \sum_{i=1}^\infty x_i(t) e_i,
\end{align*}
where each of the coefficients $x_i(t)$ is the solution to the one-dimensional geometric Brownian motion
\begin{align*}
\dd x_i(t) = -\lambda_i x_i(t) \, \dd t + \sqrt{\mu_i} x_i(t) \, \dd \beta_i(t).
\end{align*}
Furthermore, the second moment is given by
\begin{align*}%\label{eq:secondMomentInfiniteGM}
\E [\| X(T) \|_H^2 ] = \sum_{i=1}^\infty \E [ |x_i(T)|^2] = \sum_{i=1}^\infty \langle X_0,e_i \rangle_H^2\exp((-2\lambda_i + \mu_i)T).
\end{align*}
Consequently, asymptotic mean-square stability of~\eqref{eq:StochHeat} holds if and only if $-2\lambda_i + \mu_i < 0$ for all $i\in \N$. By using the explicit representation of the eigenvalues $\lambda_i$ and $\mu_i$, this corresponds to $-2 \nu i^2\pi^2 + C_\mu i^{-\alpha} < 0$ or equivalently $-2\lambda_1 + \mu_1 = -2\nu \pi^2 + C_\mu < 0$, i.e., \eqref{eq:StochHeat} is asymptotically mean-square unstable if and only if $C_\mu > 2\nu\pi^2$.

For the spectral Galerkin approximation, we choose $V_h = \operatorname{span}(e_1,\dots,e_{N_h})$, $N_h < \infty$. Thus, we consider 
% want to find $X_h(t) = \sum_{k=1}^ {N_h} \langle X(t),e_k \rangle_H e_k = \sum_{k=1}^{N_h} x_k(t)e_k$ of the semi-discretized stochastic heat equation
% \begin{align}\label{eq:semidiscreteStochHeat}
% \dd X_h(t) = A_h X_h(t) \, \dd t + P_h G_1(X_h(t)) \, \dd W(t).
% \end{align}
% where $X_h(t)$ is $V_h$-valued for all~$t$, more specifically 
$X_h(t) = \sum_{k=1}^{N_h} x_k(t)e_k$.
To obtain a fully discrete scheme, we approximate the one-dimensional geometric Brownian motions in time by the three considered rational approximations in Theorem~\ref{Thm:MSStableFD} and Proposition~\ref{Prop:MSStableMilstein}.
% : backward Euler, Crank--Nicolson, and forward Euler.
Propositions~\ref{Proposition:MSStableFD} and~\ref{Proposition:MilsteinRationalApproxS} yield asymptotic mean-square stability
% of the fully discrete approximation of~\eqref{eq:semidiscreteStochHeat} 
if and only if the corresponding linear operators $\cS$ satisfy $\rho (\cS) <1$, which is in the first case for $k,\ell = 1,\dots,N_h$ given by
% For computing the spectrum of the linear operator $\cS$ from Proposition~\ref{Proposition:MSStableFD} (Euler--Maruya\-ma scheme), we consider for $k,\ell = 1,\dots,N_h$
\begin{align*}
\cS(e_k \otimes e_\ell) 
 &= (\Ddet \otimes \Ddet)(e_k \otimes e_\ell) + \Delta t \bigl( (C\otimes C)q \bigr)(e_k \otimes e_\ell) \\
 &= (\Ddet e_k \otimes \Ddet e_\ell) + \Delta t \sum_{m=1}^\infty \mu_m \bigl(((Ce_m)e_k) \otimes ((Ce_m)e_\ell)\bigr).
\end{align*}
Since 
\begin{align*}
\Ddet e_k 
  = R(\Delta t A_h)e_k 
  = \sum_{r=1}^{N_h}R(-\Delta t \lambda_r)\langle e_k , e_r \rangle_H e_r 
  = R(-\Delta t \lambda_k)e_k
\end{align*}
and
\begin{align*}
(Ce_m)e_k 
 &= r_\mathrm{d}^{-1}(\Delta t A_h) P_h G_1(e_k)e_m 
 = r_\mathrm{d}^{-1}(\Delta t A_h) P_h\Bigl(\sum_{n=1}^\infty \langle e_k,e_n\rangle_H \langle e_m,e_n\rangle_H e_n\Bigr)\\
 &= \kdelta{k}{m} r_\mathrm{d}^{-1}(\Delta t A_h)  e_k = \kdelta{k}{m} r_\mathrm{d}^{-1}(-\Delta t \lambda_k)e_k,
\end{align*}
the corresponding eigenvalues $\Lambda_{k,\ell}$
% corresponding to the eigenfunctions $e_k \otimes e_\ell$ 
are %given by
\begin{align*}
\Lambda_{k,\ell} = R(-\Delta t \lambda_k) R(-\Delta t \lambda_\ell) + \kdelta{k}{\ell}\,\Delta t \mu_k \, r_\mathrm{d}^{-1}(-\Delta t \lambda_k)r_\mathrm{d}^{-1}(-\Delta t \lambda_\ell).
\end{align*}
Using a Milstein scheme instead, we obtain for $\cS$ in Proposition~\ref{Proposition:MilsteinRationalApproxS} with similar computations as before and the observations that the commutativity condition in Remark~\ref{rem:g_commutativity} is fulfilled and that $\Delta[\beta_k,\beta_\ell]^j =\kdelta{k}{\ell} \Delta t$
% along with similar computations to those of the Euler--Maruyama scheme, one can show that the corresponding eigenvalues $\Lambda_{k,\ell}$ are given by
\begin{align*}%\label{eq:eigenval_spectral}
\Lambda_{k,\ell} = R(-\Delta t \lambda_k) R(-\Delta t \lambda_\ell) + \kdelta{k}{\ell}\,r_\mathrm{d}^{-1}(-\Delta t \lambda_k)r_\mathrm{d}^{-1}(-\Delta t \lambda_\ell)\left( \Delta t \mu_k  + \Delta t^2\mu_k^2/2\right).
\end{align*}
% The additional quadratic term results from the linear operator $(C'\otimes C')q'$.
Note that for both operators $\cS$, the eigenvalues $\Lambda_{k,\ell}$ with $k\neq \ell$ satisfy
\begin{align*}
|\Lambda_{k,\ell}| = |R(-\Delta t \lambda_k) R(-\Delta t \lambda_\ell)| \leq R(-\Delta t \lambda_s)^2 \leq \Lambda_{s,s},
\end{align*}
where $|R(-\Delta t \lambda_s)| = \max_{j = 1,\dots,N_h}|R(-\Delta t \lambda_j)|$. Hence, $\rho(\cS) < 1$ is equivalent to $|\Lambda_{k,k}| <1$ for all $k=1,\dots,N_h$. 
In Table~\ref{tab:geometricBM_spectral} the eigenvalues $\Lambda_{k,k}$ and sufficient and necessary conditions for asymptotic mean-square stability are collected.
% of the zero solution of~\eqref{schemeEMbased} and~\eqref{Eq:MilsteinScheme} are presented for the considered rational approximations (backward Euler, Crank--Nicolson, and forward Euler).

\renewcommand{\arraystretch}{2}
\begin{table}[ht]
\small
\begin{tabular}{c||c|c}
\begin{tabular}{c}
rational approximation/\vspace{-0.4cm}\\
stochastic approximation
\end{tabular} &$\Lambda_{k,k}$ & $\rho(\cS)<1 \Leftrightarrow \text{for all }k=1,\dots,N_h:$   \\\hline
backward Euler/EM & $\frac{1 + \Delta t\mu_k}{(1+ \Delta t \lambda_k)^2}$ & $-2\lambda_k + \mu_k - \Delta t\lambda_k^2 < 0$\\\hline
backward Euler/Milstein & $\frac{1 + \Delta t\mu_k + \Delta t^2 \mu_k^2/2}{(1+ \Delta t \lambda_k)^2}$ & $-2\lambda_k + \mu_k + \Delta t(-\lambda_k^2 + \mu_k^2/2) < 0$ \\\hline
Crank--Nicolson/EM & $\frac{(1-\Delta t \lambda_k/2)^2 + \mu_k\Delta t}{(1+\Delta t \lambda_k/2)^2}$ & $-2\lambda_k + \mu_k  < 0$ \\\hline
forward Euler/EM & $(1-\Delta t \lambda_k)^2 + \mu_k\Delta t$ & $-2\lambda_k + \mu_k + \Delta t\lambda_k^2 < 0$
\end{tabular}
\caption{Spectral Galerkin methods with corresponding eigenvalues~$\Lambda_{k,k}$ and asymptotic mean-square stability conditions. 
% Euler--Maruyama is abbreviated by EM.
}
\label{tab:geometricBM_spectral}
\end{table}

As it is noted above, \eqref{eq:StochHeat} is asymptotically mean-square stable if and only if the condition $-2 \lambda_1 + \mu_1 < 0$ holds. As can be seen from Table~\ref{tab:geometricBM_spectral} and the choice of the eigenvalues, the Euler--Maruyama scheme~\eqref{schemeEMbased} with backward Euler and Crank--Nicolson rational approximation shares this property without any restriction on $V_h$ and~$\Delta t$.
\begin{figure}[ht]
\centering
\subfigure[Crank--Nicolson (CN), backward (BE) and forward (FE) Euler.\label{Fig:SpectralMaruyama_a}]{\includegraphics[width = .38\textwidth]{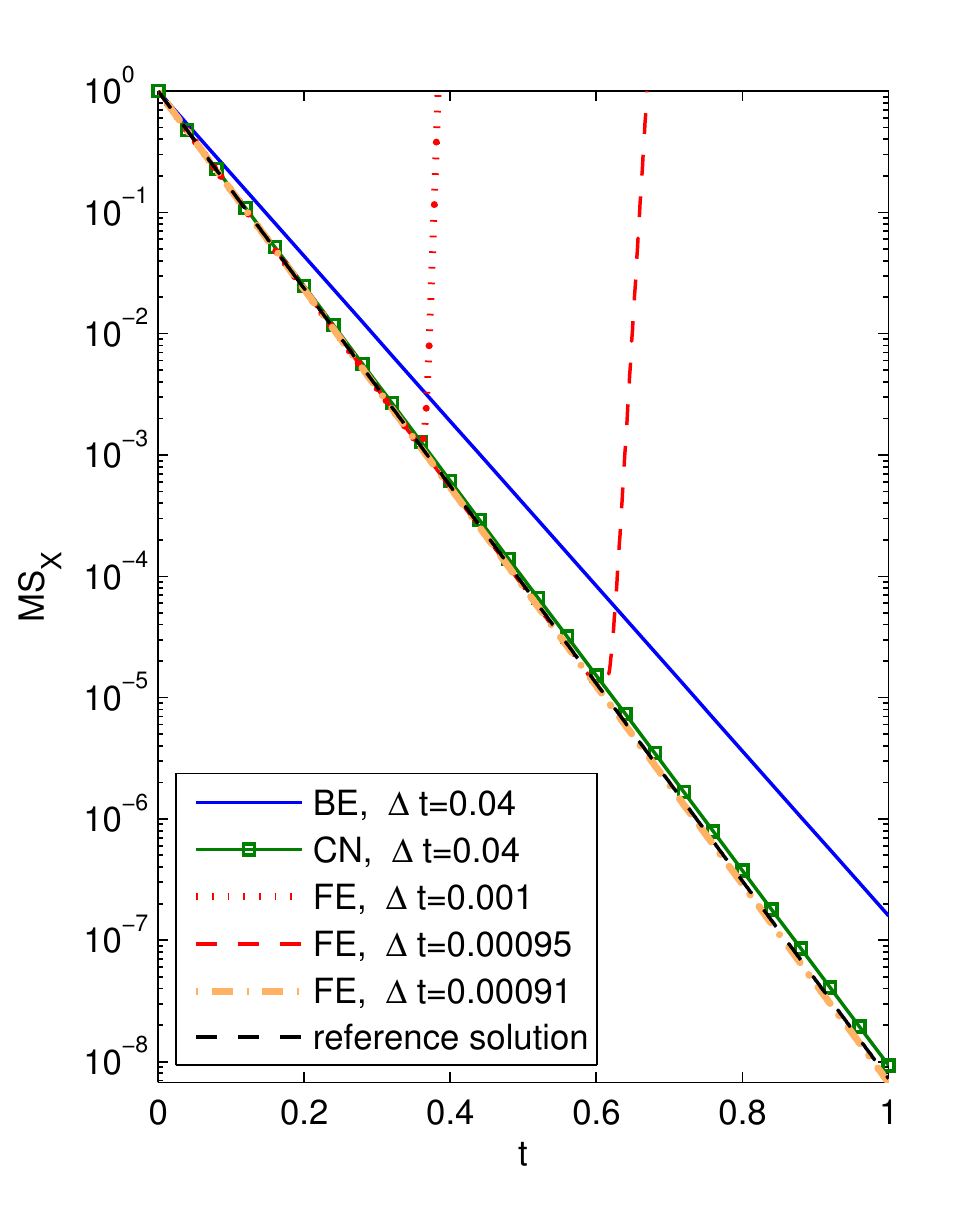}}
\hspace*{3em}
\subfigure[Euler--Maruyama (BE) and Milstein (BM) based on backward Euler.\label{Fig:SpectralMaruyama_b}]{\includegraphics[width = .38\textwidth]{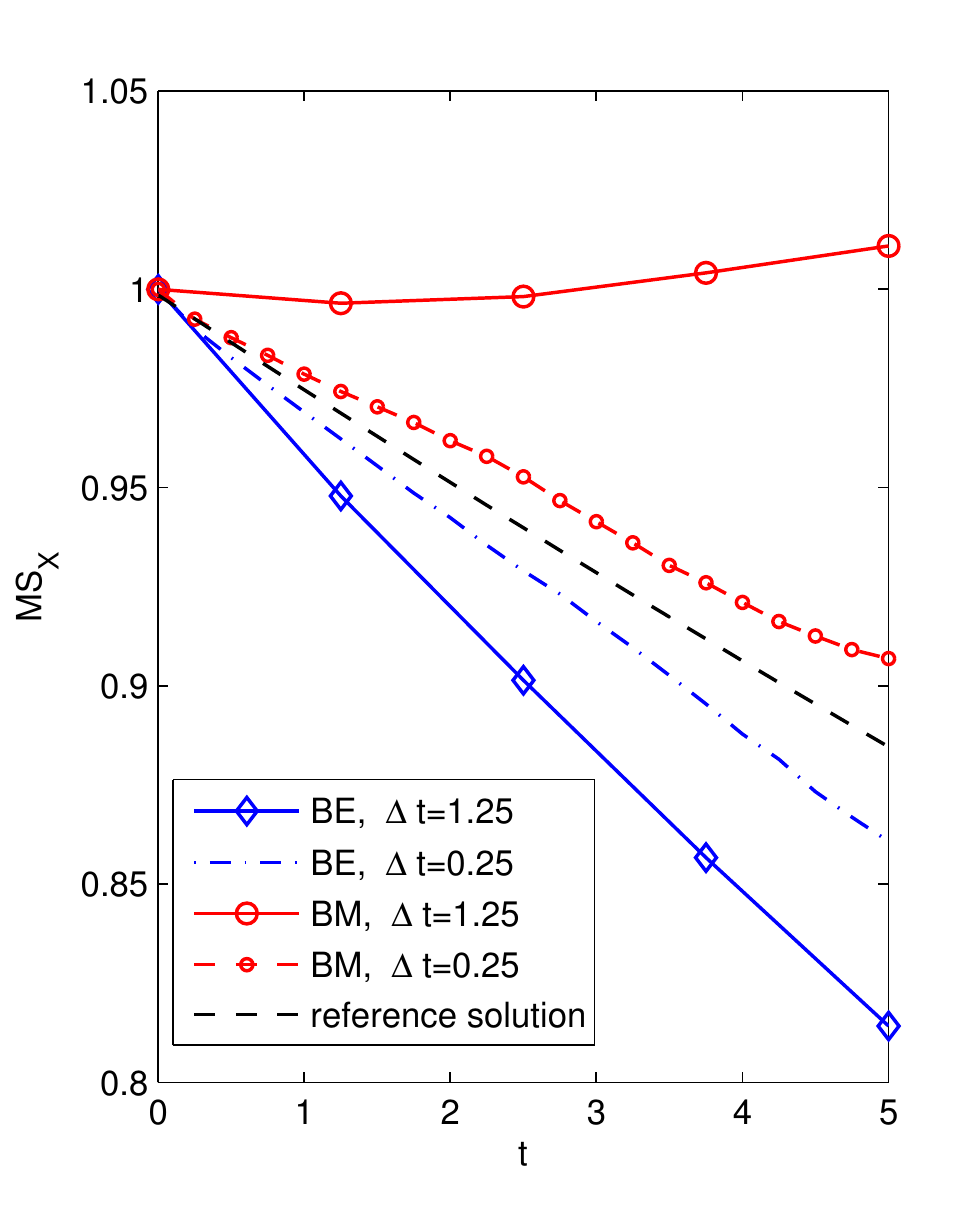}}
\caption{Spectral Galerkin approximate of~\eqref{eq:StochHeat} with~$G_1$, $N_h = 15$, and different~$\Delta t$.}
% \label{Fig:SpectralMaruyama}
\end{figure}
In Figure~\ref{Fig:SpectralMaruyama_a} the qualitative behaviour of the Euler--Maruyama method with the three rational approximations in Theorem~\ref{Thm:MSStableFD} is compared.  
We choose $\nu =1$, $N_h = 15$, and $\mu_i = i^{-3}$ for $i \in \N$, i.e., $C_\mu = 1$ and $\alpha = 3$. Since $-2\lambda_1 + C_\mu = -2\pi^2 + 1 < 0$, the analytical solution to~\eqref{eq:StochHeat} is asymptotically mean-square stable.

For the approximation of $\mathbb{E}[\|X_h^j\|_H^2]$ we use a Monte Carlo simulation with $M = 10^6$, i.e., we approximate
% independent realizations of $\|X_h^j\|^2_H$, 
% \att{\annika{What is ``$j$'' here? Confusing with the next line.}}
% which we compute by using the representation
% \begin{align}\label{eq:referenceGeoB}
% \| X_h(t) \|_H^2 = \sum_{k=1}^{N_h} |x_k(t)|^2.
% \end{align}
% More specifically, let $\text{MS}_X(t_j) \approx  \mathbb{E}[\|X_h^j\|_H^2]$ be given by
\begin{align*}
\mathbb{E}[\|X_h^j\|_H^2]
 \approx \text{MS}_X(t_j) 
 = \frac{1}{M} \sum_{i=1 }^M \sum_{k=1}^{N_h} |\widehat{x}_k^{j,(i)}|^2,
\end{align*}
where $(\widehat{x}_k^{j,(i)},i=1,\dots,M)$ consists of independent samples of numerical approximations of $x_k(t_j)$ with different schemes.
% Based on Equation~\eqref{eq:referenceGeoB}, 
The reference solution is %computed by 
\begin{align*}
\mathbb{E}[\|X_h(t) \|_H^2] = \sum_{k=1}^{N_h} \mathbb{E}[|x_k(t)|^2] = \sum_{k=1}^{N_h} \langle X_0,e_k \rangle_H^2 \exp \left( (-2\lambda_k + \mu_k)t \right).
\end{align*} 

As it can be seen in Figure~\ref{Fig:SpectralMaruyama_a}, the backward Euler and the Crank--Nicolson scheme reproduce the mean-square stability of~\eqref{eq:StochHeat}
already for large time step sizes ($\Delta t = 1/25$), but the forward Euler scheme requires a $44$~times smaller~$\Delta t$. Here, the finest time step size is given by $\Delta t = 1/1100$  which satisfies the restrictive bound in Table~\ref{tab:geometricBM_spectral} such that $\rho(\cS) < 1$. Due to a rapid amplification of oscillations caused by negative values of $X_h^j$ for coarser time step sizes outside the stability region ($\Delta t = 1/1000$ and $1/1050$), the mean-square process deviates rapidly from the reference solution at a certain time point.

In Figure~\ref{Fig:SpectralMaruyama_b} the qualitative behaviour of the Euler--Maruyama and the Milstein scheme with a backward Euler rational approximation on the time interval~$[0,5]$ are compared. The parameters $\nu = 8/(5\pi^4)$
% (small diffusion) 
and $\mu_i = 3/10\, i^{-3}$ are chosen such that the Milstein scheme is asymptotically mean-square unstable for $\Delta t = 1.25$ and asymptotically mean-square stable for $\Delta t = 0.25$ while the Euler--Maruyama scheme is asymptotically mean-square stable for both choices. These theoretical results are reproduced in the simulation.

\subsection{Galerkin finite element methods}

Let us continue with $G = G_2$ in Example~\ref{Ex:stochastic_heat_equation} and a Galerkin finite element setting, similar to that of \cite{LP17}. This is to say, we let $V_h$ be the span of piecewise linear functions on an equidistant grid of $[0,1]$ with $N_h$ interior nodes so that $V_h$ is an $N_h$-dimensional subspace of~$\dot{H}^1$ with refinement parameter $h=(N_h+1)^{-1}$. With the exception that $U=\dot{H}^1$, all other parameters are as in Figure~\ref{Fig:SpectralMaruyama_a} of Section~\ref{subsec:spectral_galerkin}. 

In contrast to the setting in Section~\ref{subsec:spectral_galerkin}, the solution and its approximation are no longer sums of one-dimensional geometric Brownian motions and thus, analytical necessary and sufficient conditions for $\rho(\cS) < 1$ are not available. We therefore consider the results of Theorem~\ref{Thm:MSStableFD} instead.
% . The identity $\lambda_{h,i} = \lambda_i$ for $i \in \N$ does not hold, but an analysis in \cite[Section 6.1]{K14} shows that $\lambda_{h,i}$ is explicitly given by 
With the setting of this section,
\begin{align*}
\lambda_{h,i} = 4 \nu h^{-2} 3 \left( 2 + \cos(i \pi h) \right)^{-1} \left(  \sin(i \pi h / 2) \right)^2
\end{align*}
for $i \in \N$, which was derived in~\cite[Section 6.1]{K14}. For the convenience of the reader, the sufficient conditions of Theorem~\ref{Thm:MSStableFD} for the considered approximation schemes are collected in simplified form in Table~\ref{tab:nemytskii_fem}, expressed in terms of stability parameters $\rho_{\text{BE}}$, $\rho_{\text{CN}}$ and $\rho_{\text{FE}}$. 
By setting
% \begin{align*}
	$\hat{g} = \Bigl( 2 \sum^\infty_{i=1} \lambda_i^{-1} \Bigr)^{1/2}$,
% \end{align*}
we replace $\norm{G_2}{L(H;L(U;H))}$ in these conditions with the upper bound derived in Example~\ref{Ex:stochastic_heat_equation}. Note that Corollary~\ref{cor:stability_connection} with these parameters implies simultaneous asymptotic mean-square stability of~\eqref{eq:StochHeat} and the finite element backward Euler scheme~\eqref{schemeEMbased}.

\begin{table}[ht!]
	\small
	\begin{tabular}{c||c}
		\begin{tabular}{c}
			rational approximation
		\end{tabular} & $\rho(\cS)<1 \Leftarrow:$   \\\hline
		backward Euler & $\rho_{\text{BE}} = \Delta t \trace(Q) \hat{g}^2 -  2 \Delta t \lambda_{h,1} - \Delta t^2 \lambda_{h,1}^2 < 0$\\\hline
		Crank--Nicolson & $\rho_{\text{CN}} = \max\limits_{k \in \{ 1,N_h\}} \left| \frac{1- \Delta t \lambda_{h,k}/2}{1 + \Delta t \lambda_{h,k}/2} \right|^2  +  \frac{\Delta t \trace(Q) \hat{g}^2}{(1 + \Delta t \lambda_{h,1}/2)^2} -1 < 0 $ \\\hline
		forward Euler & $\rho_{\text{FE}} = \max\limits_{k \in \{ 1,N_h\}} (1- \Delta t \lambda_{h,k})^2 + \Delta t \trace(Q) \hat{g}^2 -1 < 0$
	\end{tabular}
	\caption{Finite element methods with sufficient conditions for $\rho(\cS) < 1$.}
	\label{tab:nemytskii_fem}
\end{table}

As in Section~\ref{subsec:spectral_galerkin} we compare the mean-square behaviour of the backward Euler and the forward Euler scheme in Figure~\ref{Fig:FEM_a} but now for the finite element discretization up to $T=2.5$. We observe that the increase of the time step size by a very small amount, i.e., from $\Delta t = 0.00066$ to $\Delta t = 0.00067$, causes the forward Euler system to switch from a stable to an unstable behaviour. This agrees with the theory in Table~\ref{tab:nemytskii_fem_2}, as $\rho_{\text{FE}}$ changes sign in that interval, i.e., stability is only guaranteed for the smaller time step. Therefore we conclude that the sufficient condition is sharp in our model problem. 

For the approximation of $\mathbb{E}[\|X_h^j\|_H^2]$, we use the same method as before but take $M=10^4$ samples in the Monte Carlo approximation. For the computation of the norm in $H$, we use the fact that for given representation
% \begin{align*}
$X_h^j = \sum_{m=1}^{N_h} x_m \phi_m$
% \end{align*}
with respect to the \emph{hat functions} $\{\phi_m, m=1\ldots,N_h\}$ that span~$V_h$
\begin{align*}
	\|X_h^j\|_H^2 = \sum_{m=1}^{N_h} \sum_{n=1}^{N_h} x_m x_n \inpro{\phi_m}{\phi_n}{H}.
\end{align*}

In Figure~\ref{Fig:FEM_b} the mean-square behaviour of the the backward Euler scheme and the Crank--Nicolson scheme for $\Delta t = 0.015$ to $\Delta t = 0.15$ is compared. We see from Table~\ref{tab:nemytskii_fem_2} that $\rho_{\text{CN}}$ changes sign when the time step size is increased, which occurs for significantly larger time steps than for the forward Euler scheme. The simulation results show a substantial change in the decay behaviour of $\E [ \| X_h^j \|_H^2]$ for the Crank--Nicolson scheme with time step size $\Delta t = 0.15$ compared to $\Delta t = 0.015$, which is no longer convincing to be mean-square stable. Since the sufficient condition $\rho_{CN} < 0$ from Table~\ref{tab:nemytskii_fem} is not fulfilled for $\Delta t = 0.15$, it is unclear from the theory if asymptotic mean-square stability holds in that case.

\begin{figure}[ht]
	\centering
	\subfigure[Backward (BE) and forward (FE) Euler.\label{Fig:FEM_a}]{\includegraphics[width = .38\textwidth]{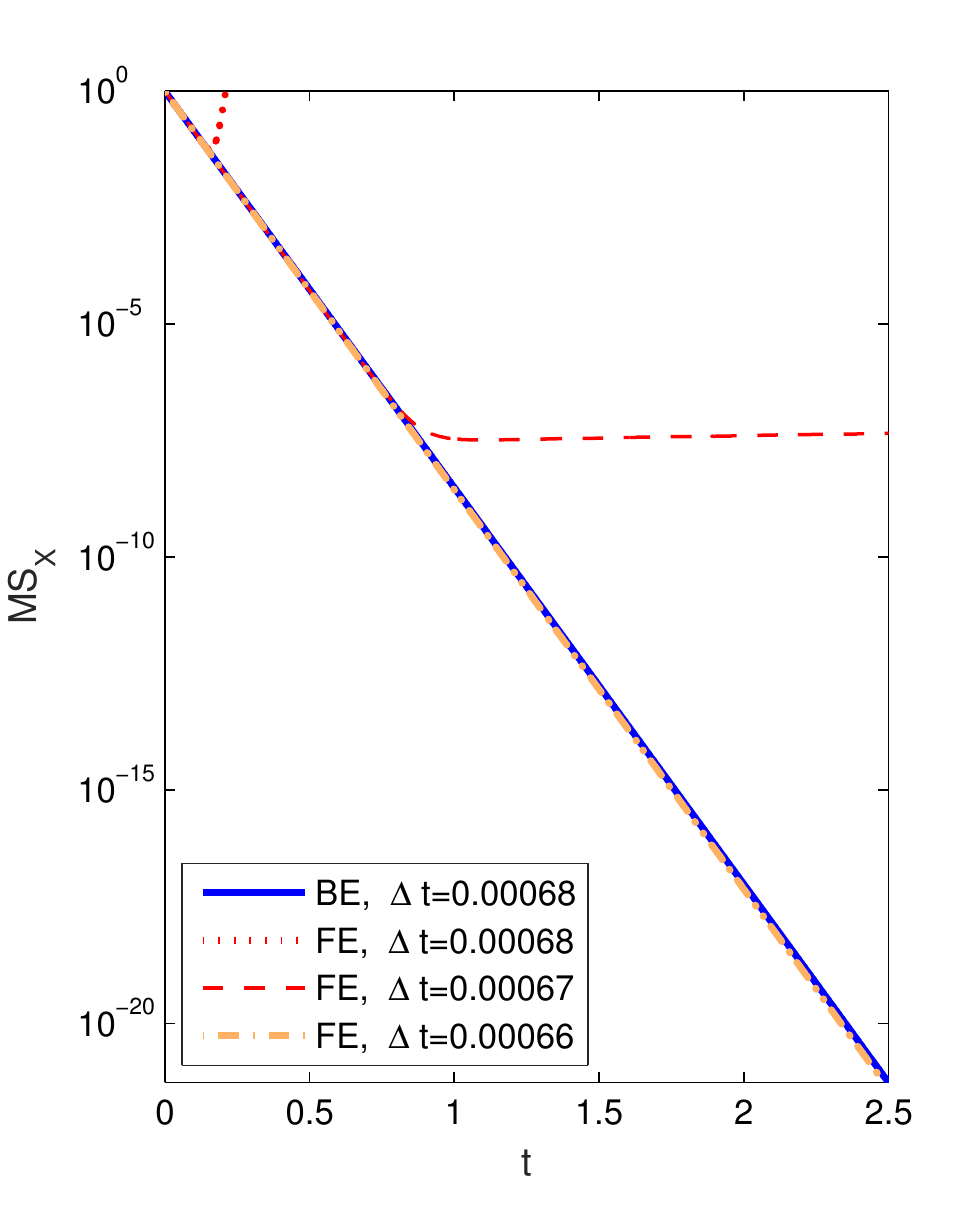}}
	\hspace*{3em}
	\subfigure[Backward Euler (BE) and Crank--Nicolson (CN).\label{Fig:FEM_b}]{\includegraphics[width = .38\textwidth]{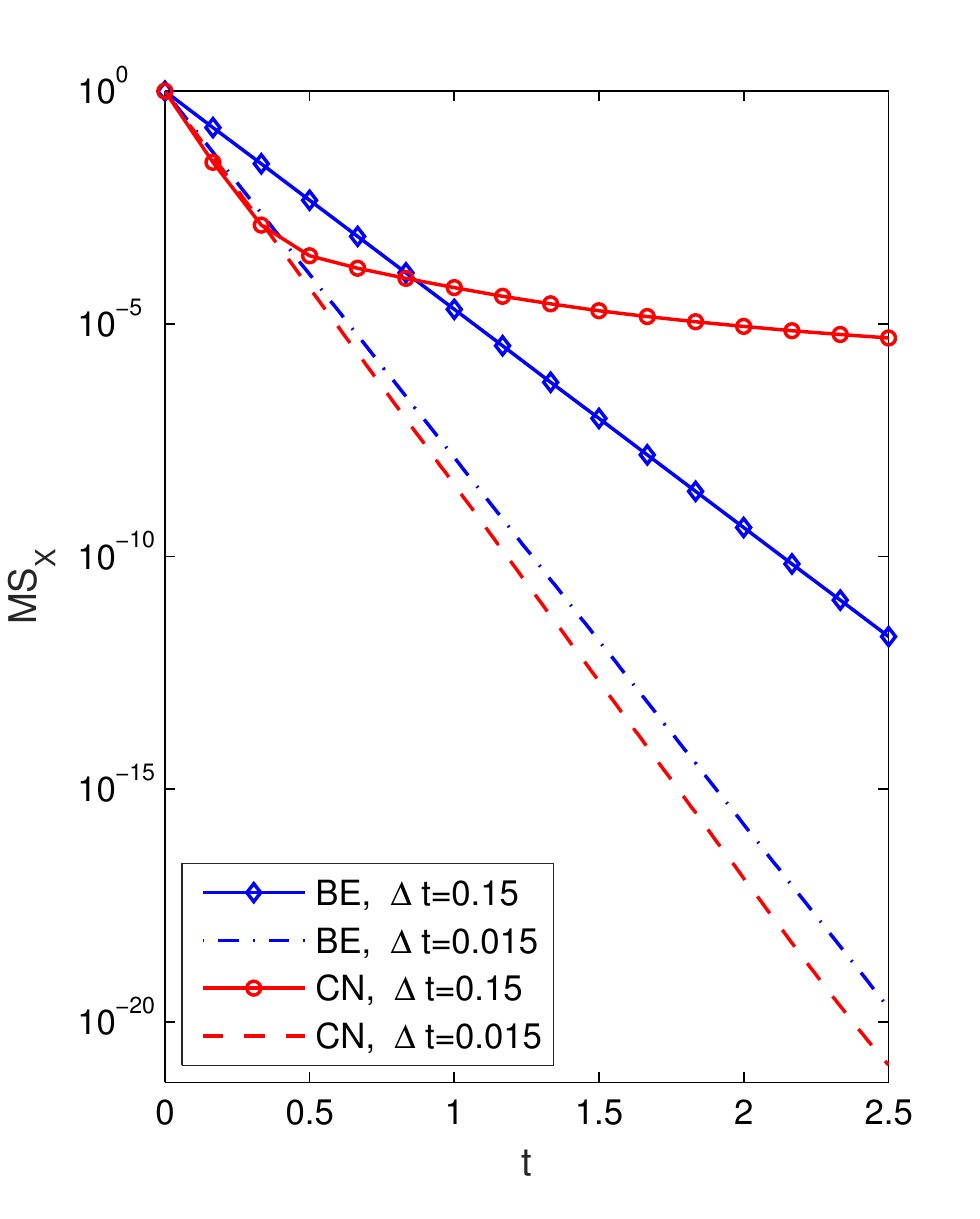}}
	\caption{Finite Element approximation of~\eqref{eq:StochHeat} with $G=G_2$, $N_h = 15$, and different~$\Delta t$.}
% 	\label{Fig:FEM}
\end{figure}

\begin{table}[ht!]
	\small
	\begin{tabular}{c||c|c|c}
		\begin{tabular}{c}
		\end{tabular}
		\backslashbox{$\Delta t$}{ } & $\rho_{\text{BE}}$ & $\rho_{\text{CN}}$ & $\rho_{\text{FE}}$ \\\hline
				0.15 & -5.11613e+00 & 2.08460e-03 & 1.99602e+05 \\\hline
				0.015 & -3.13089e-01 & -1.58504e-01 & 1.91542e+03 \\\hline
				0.00068 & -1.32387e-02 & -1.31050e-02 & 6.09395e-02 \\\hline
				0.00067 & -1.30434e-02 & -1.29135e-02 & 3.39709e-04 \\\hline
				0.00066 & -1.28480e-02 & -1.27221e-02 & -1.27626e-02 
	\end{tabular}
	\caption{Specific values of Table~\ref{tab:nemytskii_fem} for varying~$\Delta t$.}
	\label{tab:nemytskii_fem_2}
\end{table}

\appendix 
\section{Properties of \levy increments}

In this appendix we derive properties of the $U$-valued, square-integrable \levy process that are used in the proofs of Propositions~\ref{Proposition:MSStableFD} and~\ref{Proposition:MilsteinRationalApproxS}. We apply the same setting and notation as in Section~\ref{sec:applications-galerkin}.
% For this, we recall that the \KL expansion of $L$ for $t \ge 0$ is given by 
% \begin{align*}
% L(t) = \sum_{k=1}^\infty \sqrt{\mu_k} L_k (t) f_k,
% \end{align*}
% where $(L_k, k \in \N)$ is a sequence of real-valued, uncorrelated, zero-mean \levy processes that fulfil for all $k \in \N$, $\E[L^2_k(t)] = t$. Here, $(\mu_k, k \in \N)$ is the sequence of eigenvalues of the covariance operator $Q$ of $L$ and $(f_k, k \in \N)$ is an orthonormal eigenbasis of~Q.

\begin{lemma}
\label{lem:appendix:em}
Let $L$ be a $U$-valued \levy process and let, for $0 \le a < b$, $\Delta L = L(b) - L(a)$ and $\Delta t = b-a$. Then 
\begin{align*}
\E [\Delta L \otimes \Delta L] = \Delta t \sum_{k=1}^{\infty} \mu_k f_k \otimes f_k.
\end{align*}
\end{lemma}

\begin{proof}
We first note that $\Delta L \otimes \Delta L$ is well-defined as a member of $L^1(\Omega;U^{(2)})$ since 
\begin{align*}
\E [\norm{\Delta L \otimes \Delta L}{U^{(2)}}] = \E [\norm{\Delta L}{U}^2] = \trace(Q) \Delta t < \infty.
\end{align*}
The increments $\Delta L_k = L_k(b) - L_k(a)$ of 
% \begin{align*}
$\Delta L = \sum_{k=1}^\infty \sqrt{\mu_k} \Delta L_k f_k$
% \end{align*}
fulfil $\E[\Delta L_k \Delta L_\ell] =  \kdelta{k}{\ell} \Delta t $ for $k, \ell \in \N$. Thus, we obtain
\begin{equation*}
\E [\Delta L \otimes \Delta L] = \sum_{k,\ell = 1}^\infty \sqrt{\mu_k \mu_l} \E[\Delta L_k \Delta L_l] \, f_k \otimes f_\ell = \Delta t \sum_{k=1}^{\infty} \mu_k f_k \otimes f_k.
\qedhere
\end{equation*}
\end{proof}

% The following lemma is used in Section~\ref{subsec:TimeRational} to find a suitable representation of the linear operator $\cS$ for the backward Euler--Milstein scheme.

\begin{lemma}
\label{lem:appendix:mil}
Let $L$ be a $U$-valued, square-integrable \levy process and set for $0 \le a < b$ with $\Delta t = b-a$,
\begin{align*}
\Delta^{(2)} L 
  = \sum_{k,\ell=1}^\infty \sqrt{\mu_k \mu_\ell} \Bigl( \int^b_a \int^s_a \, \dd L_k(r) \, \dd L_\ell(s)  \Bigr) f_k \otimes f_{\ell} 
    \in L^2(\Omega;U^{(2)}).
\end{align*}
Then 
\begin{enumerate}
\item $\E\left[\Delta^{(2)} L \otimes \Delta L\right] = 0$,
\item $\E\left[\Delta^{(2)} L \otimes \Delta^{(2)} L \right] = \frac{\Delta t^2}{2} \sum_{k,\ell = 1}^\infty \mu_k \mu_\ell
\big( (f_k \otimes f_\ell) \otimes (f_k \otimes f_\ell) \big)$.
\end{enumerate}
\end{lemma}

\begin{proof}
Since $L$ is stationary, we may assume without loss of generality that $a=0$ and $b=t>0$. We first note that
\begin{align*}
&\E\Bigl[ \Bigl( \int^t_0 \int^s_0 \, \dd L_i(r) \, \dd L_j(s)  \Bigr) \Bigl( \int^t_0 \int^s_0 \, \dd L_k(r) \, \dd L_\ell(s)  \Bigr) \Bigr] \\ 
& \hspace*{5em} = \E\Bigl[ \Bigl( \int^t_0 L_i(s-) \, \dd L_j(s)  \Bigr) \Bigl( \int^t_0 L_k(s-) \, \dd L_\ell(s)  \Bigr) \Bigr].
\end{align*}
To simplify this expression,
we use the angle bracket process $(\inpro{X}{Y}{t}, t \ge 0)$, which for two real-valued semimartingales $X$ and $Y$ with (locally) integrable quadratic covariation $[X,Y]$ is defined as the unique compensator which makes $([X,Y]_t-\inpro{X}{Y}{t}, t \ge 0)$ a local martingale. For this, we have the \emph{polarization identity},
\begin{align*}
\inpro{X}{Y}{t} = \frac{1}{4} \left( \inpro{X+Y}{X+Y}{t} - \inpro{X-Y}{X-Y}{t} \right),
\end{align*}
which can be found, along with an introduction to this process, e.g., in~\cite[Section III.5]{P04}.

For square-integrable martingales $M$, it holds (see, e.g.,~\cite[Section 8.9]{K12b}) that $\E[\inpro{M}{M}{t}] = \E[M^2(t)]$ and therefore, by the polarization identity, if $N$ is another square-integrable martingale, then,
\begin{align*}
\E[\inpro{M}{N}{t}] = \frac{1}{4} \left( \E[(M(t)+N(t))^2] - \E[(M(t)-N(t) )^2] \right) = \E[M(t) N(t)] .
\end{align*}
Applying this to the \levy integral, which is a martingale, we obtain 
\begin{align*}
& \E\Bigl[ \Bigl( \int^t_0 L_i(s-) \, \dd L_j(s)  \Bigr) \Bigl( \int^t_0 L_k(s-) \, \dd L_\ell(s)  \Bigr) \Bigr] \\ 
& \qquad = \E\Bigl[ \bigl\langle \int_0 L_i(s-) \, \dd L_j(s),\int_0 L_k(s-) \, \dd L_\ell(s)\bigr\rangle_t   \Bigr] 
= \E\Bigl[ \int^t_0 L_i(s-)L_k(s-) \, \dd \inpro{L_j}{L_\ell}{s} \Bigr],
\end{align*}
where the last equality is a property of the angle bracket process and the stochastic integral, see \cite[Section 8.9]{K12b}. Now, when $j = \ell$, we have, since $L_j$ is a \levy process and $\E[L_j^2(s)] = s$, that $\inpro{L_j}{L_\ell}{s} = \inpro{L_j}{L_j}{s} = s$ by \cite[Chapter 8]{PZ07}. When $j \neq \ell$ on the other hand, $L_j L_\ell$ is a square-integrable martingale by \cite[Theorem 4.49(ii)]{PZ07}. Integration by parts yields
\begin{align*}
	[L_j,L_\ell]_s = L_j(s) L_\ell(s) - \int^s_0 L_j(r-) \, \dd L_\ell(r) - \int^s_0 L_\ell(r-) \, \dd L_j(r).
\end{align*}
Therefore, $[L_j,L_\ell]$ is also a square-integrable martingale (with zero mean), because the right hand side is a square-integrable martingale.
Since $(\inpro{L_j}{L_\ell}{s}, s \ge 0)$ is the \emph{unique} compensator of $[L_j,L_\ell]$ it must follow that $\inpro{L_j}{L_\ell}{s} = 0$ for all $s \ge 0$. Thus, 
% \begin{align*}
	$\E [ \int^t_0 L_i(s-)L_k(s-) \, \dd \inpro{L_j}{L_\ell}{s} ]$
% \end{align*} 
is non-zero only if $j = \ell$, and in that case
\begin{align*}
\E\Bigl[ \int^t_0 L_i(s-)L_k(s-) \, \dd \inpro{L_j}{L_j}{s} \Bigr] 
 =  \int^t_0 \E\left[L_i(s-)L_k(s-)\right] \, \dd s.
\end{align*}
In conclusion we have obtained
\begin{equation}
\label{eq:iteratedintegralcovariance}
	\E\Bigl[ \Bigl( \int^t_0 \int^s_0 \, \dd L_i(r) \, \dd L_j(s)  \Bigr) \Bigl( \int^t_0 \int^s_0 \, \dd L_k(r) \, \dd L_\ell(s)  \Bigr) \Bigr] =
	\begin{dcases*} 
	t^2/2 & for $j = \ell$ and $i = k$, \\
	0 & otherwise,
	\end{dcases*}
\end{equation}
which yields by the monotone convergence theorem that $\Delta^{(2)} L \in L^2(\Omega;U^{(2)})$ with
\begin{align*}
\E\Bigl[\norm{\Delta^{(2)} L}{U^{(2)}}^2\Bigr]  
= \sum_{k,\ell=1}^\infty \mu_k \mu_\ell \E\Bigl[\Bigl( \int^t_0 \int^s_0 \, \dd L_k(r) \, \dd L_\ell(s)  \Bigr)^2\Bigr] 
= \frac{t^2}{2} \sum_{k,\ell=1}^\infty \mu_k \mu_\ell = \frac{t^2}{2} \trace(Q)^2 
< \infty.
\end{align*}
This entails that $\Delta^{(2)} L \otimes \Delta L \in L^1(\Omega;U^{(2)}\otimes U)$, since 
\begin{align*}
\left(\E\bigl[\norm{\Delta^{(2)} L \otimes \Delta L}{U^{(2)} \otimes U}\bigr]\right)^{2} 
 & = \left(\E\bigl[\norm{\Delta^{(2)} L}{U^{(2)}} \norm{\Delta L}{U}\bigr]\right)^{2} \\ 
 & \le \E\bigl[\norm{\Delta^{(2)} L}{U^{(2)}}^2\bigr] \E\left[\norm{\Delta L}{U}^2\right] 
 = \frac{t^3}{2} \trace(Q)^3 
    < \infty.
\end{align*}
by the Cauchy--Schwarz inequality.
Similarly, it holds that $\Delta^{(2)} L \otimes \Delta ^{(2)}L \in L^1(\Omega;U^{(2)}\otimes U^{(2)})$. 
Therefore, we obtain
\begin{align*}
 \E\bigl[\Delta^{(2)} L \otimes \Delta L\bigr] 
    = \sum_{k,\ell,m=1}^\infty \sqrt{\mu_k \mu_\ell \mu_m} \;\E\Bigl[\Delta L_m\Bigl( \int^t_0 \int^s_0 \, \dd L_k(r) \, \dd L_\ell(s)  \Bigr)\Bigr] (f_k \otimes f_\ell) \otimes f_m,
\end{align*}
and, in the same way as the first observation of this proof,
\begin{align*}
\E\Bigl[\Delta L_m\Bigl( \int^t_0 \int^s_0 \, \dd L_k(r) \, \dd L_\ell(s)  \Bigr)\Bigr] 
 & = \E\Bigl[\bigl\langle \int \, \dd L_m(s),\int L_k(s-) \, \dd L_\ell(s)\bigr\rangle_t \Bigr] \\ 
 &= \E\Bigl[\int^t_0 L_k(s-) \, \dd  \inpro{L_m}{L_\ell}{s}\Bigr] 
 = 0.
\end{align*}
This is justified since $\inpro{L_m}{L_\ell}{s} \neq 0$ only if $m = \ell$ and that in this case the expectation of the integral is still zero since $L_k$ has zero expectation. 

We note that by~\eqref{eq:iteratedintegralcovariance}
\begin{align*}
\E\bigl[\Delta^{(2)} L \otimes \Delta^{(2)} L \bigr] 
&= \sum_{i,j,k,\ell = 1}^\infty \sqrt{\mu_i \mu_j \mu_k \mu_\ell} \big( (f_i \otimes f_j) \otimes (f_k \otimes f_\ell) \big) \\ 
&\hspace*{5em} \cdot \E\Bigl[ \Bigl( \int^t_0 \int^s_0 \, \dd L_i(r) \, \dd L_j(s)  \Bigr) \Bigl( \int^t_0 \int^s_0 \, \dd L_k(r) \, \dd L_\ell(s)  \Bigr) \Bigr] \\
&= \frac{t^2}{2} \sum_{k,\ell = 1}^\infty \mu_k \mu_\ell
\big( (f_k \otimes f_\ell) \otimes (f_k \otimes f_\ell) \big),
\end{align*}
which shows the second claim.
\end{proof}

\bibliographystyle{hplain}
\bibliography{ms-stability}

\begin{thebibliography}{10}

\bibitem{A13}
Assyr Abdulle and Adrian Blumenthal.
\newblock Stabilized multilevel {M}onte {C}arlo method for stiff stochastic
  differential equations.
\newblock {\em J. Comput. Phys.}, 251:445--460, 2013.

\bibitem{A74}
Ludwig Arnold.
\newblock {\em Stochastic Differential Equations: Theory and Applications}.
\newblock John Wiley \& Sons, 1974.
\newblock Translated from the German.

\bibitem{BL12}
Andrea Barth and Annika Lang.
\newblock Milstein approximation for advection-diffusion equations driven by
  multiplicative noncontinuous martingale noises.
\newblock {\em Appl. Math. Optim.}, 66(3):387--413, 2012.

\bibitem{BL12a}
Andrea Barth and Annika Lang.
\newblock Multilevel {M}onte {C}arlo method with applications to stochastic
  partial differential equations.
\newblock {\em Int. J. Comput. Math.}, 89(18):2479--2498, 2012.

\bibitem{BL12c}
Andrea Barth and Annika Lang.
\newblock Simulation of stochastic partial differential equations using finite
  element methods.
\newblock {\em Stochastics}, 84(2-3):217--231, 2012.

\bibitem{BLS13}
Andrea Barth, Annika Lang, and Christoph Schwab.
\newblock Multilevel {M}onte {C}arlo method for parabolic stochastic partial
  differential equations.
\newblock {\em BIT}, 53(1):3--27, 2013.

\bibitem{BS16}
Andrea Barth and Andreas Stein.
\newblock Approximation and simulation of infinite-dimensional l\'evy
  processes.
\newblock arXiv:1612.05541 [math.PR], December 2016.

\bibitem{B64}
Richard Bellman.
\newblock Stochastic transformations and functional equations.
\newblock In {\em Proc. {S}ympos. {A}ppl. {M}ath., {V}ol. {XVI}}, pages
  171--177. Amer. Math. Soc., 1964.

\bibitem{B10}
Daniele Boffi.
\newblock Finite element approximation of eigenvalue problems.
\newblock {\em Acta Numer.}, 19:1--120, 2010.

\bibitem{BK10}
Evelyn Buckwar and C\'onall Kelly.
\newblock Towards a systematic linear stability analysis of numerical methods
  for systems of stochastic differential equations.
\newblock {\em SIAM J. Num. Anal.}, 48(1):298--321, 2010.

\bibitem{BS12}
Evelyn Buckwar and Thorsten Sickenberger.
\newblock A structural analysis of asymptotic mean-square stability for
  multi-dimensional linear stochastic differential systems.
\newblock {\em Appl. Numer. Math.}, 62(7):842--859, 2012.

\bibitem{BW06}
Evelyn Buckwar and Renate Winkler.
\newblock Multistep methods for {SDE}s and their application to problems with
  small noise.
\newblock {\em SIAM J. Numer. Anal.}, 44(2):779--803, 2006.

\bibitem{DHP12}
Thomas Dunst, Erika Hausenblas, and Andreas Prohl.
\newblock Approximate {E}uler method for parabolic stochastic partial
  differential equations driven by space-time {L}\'evy noise.
\newblock {\em SIAM J. Numer. Anal.}, 50(6):2873--2896, 2012.

\bibitem{H00a}
Desmond~J. Higham.
\newblock {$A$}-stability and stochastic mean-square stability.
\newblock {\em BIT}, 40(2):404--409, 2000.

\bibitem{H00}
Desmond~J. Higham.
\newblock Mean-square and asymptotic stability of the stochastic theta method.
\newblock {\em SIAM J. Numer. Anal.}, 38(3):753--769 (electronic), 2000.

\bibitem{JR15}
Arnulf Jentzen and Michael R\"ockner.
\newblock A {M}ilstein scheme for {SPDE}s.
\newblock {\em Found. Comput. Math.}, 15(2):313--362, 2015.

\bibitem{K12}
Rafail Khasminskii.
\newblock {\em Stochastic Stability of Differential Equations}, volume~66 of
  {\em Stochastic Modelling and Applied Probability}.
\newblock Springer, 2012.

\bibitem{KLL17}
Kristin Kirchner, Annika Lang, and Stig Larsson.
\newblock Covariance structure of parabolic stochastic partial differential
  equations with multiplicative {L}\'evy noise.
\newblock {\em J. Diff. Equations}, 262, No. 12(12):5896--5927, June 2017.

\bibitem{K12b}
Fima~C. Klebaner.
\newblock {\em Introduction to Stochastic Calculus with Applications}.
\newblock Imperial College Press, 3rd edition, 2012.

\bibitem{K14}
Raphael Kruse.
\newblock {\em Strong and Weak Approximation of Semilinear Stochastic Evolution
  Equations}, volume 2093 of {\em Lecture Notes in Mathematics}.
\newblock Springer, 2014.

\bibitem{L10}
Annika Lang.
\newblock A {L}ax equivalence theorem for stochastic differential equations.
\newblock {\em J. Comput. Appl. Math.}, 234(12):3387--3396, 2010.

\bibitem{LP17}
Annika Lang and Andreas Petersson.
\newblock {M}onte {C}arlo versus multilevel {M}onte {C}arlo in weak error
  simulations of {SPDE} approximations.
\newblock {\em Math. Comp. in Simulation}, May 2017.

\bibitem{LS15}
Annika Lang and {Ch}ristoph Schwab.
\newblock Isotropic {G}aussian random fields on the sphere: regularity, fast
  simulation and stochastic partial differential equations.
\newblock {\em Ann. Appl. Probab.}, 25(6):3047--3094, 2015.

\bibitem{L06}
Kai Liu.
\newblock {\em Stability of Infinite Dimensional Stochastic Differential
  Equations with Applications}, volume 135 of {\em Monographs and Surveys in
  Pure and Applied Mathematics}.
\newblock Chapman \& Hall/CRC, 2006.

\bibitem{LPS14}
Gabriel~J. Lord, Catherine~E. Powell, and Tony Shardlow.
\newblock {\em An Introduction to Computational Stochastic PDEs}.
\newblock Cambridge Texts in Applied Mathematics. Cambridge University Press,
  2014.

\bibitem{M07}
Xuerong Mao.
\newblock {\em Stochastic Differential Equations and Applications}.
\newblock Horwood Publishing Limited, 2008.

\bibitem{PZ07}
Szymon Peszat and Jerzy Zabczyk.
\newblock {\em Stochastic Partial Differential Equations with L\'evy Noise. An
  Evolution Equation Approach}, volume 113 of {\em Encyclopedia of Mathematics
  and Its Applications}.
\newblock Cambridge University Press, 2007.

\bibitem{P04}
Philip~E. Protter.
\newblock {\em Stochastic Integration and Differential Equations}.
\newblock Springer, 2nd edition, 2004.

\bibitem{SM96}
Yoshihiro Saito and Taketomo Mitsui.
\newblock Stability analysis of numerical schemes for stochastic differential
  equations.
\newblock {\em SIAM J. Numer. Anal.}, 33(6):2254--2267, 1996.

\bibitem{T06}
Vidar Thom{\'e}e.
\newblock {\em Galerkin Finite Element Methods for Parabolic Problems},
  volume~25 of {\em Springer Series in Computational Mathematics}.
\newblock Springer, 2nd edition, 2006.

\end{thebibliography}

\end{document}